\documentclass[11pt, reqno]{amsart}
\usepackage{fullpage}
\usepackage{mathrsfs}
\usepackage{amssymb}
\usepackage{amsfonts}
\usepackage{amsmath, color,array}
 \usepackage{float}
 \usepackage[colorlinks,linkcolor=blue,urlcolor=cyan,citecolor=blue]{hyperref}
\usepackage{tikz}
\usepackage{multirow,graphicx}
\usepackage{braket}
\usepackage{cleveref}
\usepackage{caption}
\usepackage{longtable,booktabs}
\usepackage{arydshln}



\makeatletter
\def\inn#1#2{\left\langle
      \def\ta{#1}\def\tb{#2}
      \ifx\ta\@empty{\;} \else {\ta}\fi ,
      \ifx\tb\@empty{\;} \else {\tb}\fi
      \right\rangle}
\makeatother

\def\fgg{\mathfrak{g}}

\def\cC{\mathcal{C}}

\def\mystrut{\vbox to .15in{}}

\newcommand{\al}{\alpha}

\newcommand{\frg}{\mathfrak{g}}
\newcommand{\frh}{\mathfrak{h}}

\newcommand{\fru}{\mathfrak{u}}

\newcommand{\caC}{\mathcal{C}}

\newcommand{\hobox}[3]{\draw (0+#1,0-#2) rectangle (1+#1,-1-#2)++(-0.5,+0.5) node {$ #3$};}

\newcommand{\mc}[1]{\mathcal{#1}}

\newcommand{\rar}{\rightarrow}

\newcommand{\Ann}{\operatorname{Ann}}

\newtheorem{Thm}{Theorem}[section]
\newtheorem{cor}[Thm]{Corollary}
\newtheorem{prop}[Thm]{Proposition}
\newtheorem{lem}[Thm]{Lemma}

\theoremstyle{definition}
\newtheorem{Rem}[Thm]{Remark}
\newtheorem{dfn}[Thm]{Definition}

\numberwithin{equation}{section}

\newcommand{\trivial}[2][]{\if\relax\detokenize{#1}\relax{\color{red}\vspace{0em} $[$  #2 $]$}\else
\ifx#1h
\ifcsname showtrivial\endcsname{\color{orange} \vspace{0em}  $[$ #2 $]$}\fi\else{\colorWrong argument!}\fi\fi\ignorespaces}

\usepackage{enumitem}
\setlist[enumerate,1]{font=\textup,
	leftmargin=5mm,labelsep=0.2em,topsep=1mm,itemsep=1mm,itemindent=1em,listparindent=1em}
\setlist[enumerate,2]{font=\textup, leftmargin=3mm,labelsep=0.2mm,topsep=1mm,itemsep=1mm,itemindent=1em}
\setlist[enumerate,3]{font=\textup, leftmargin=4mm,labelsep=1mm,topsep=1mm,itemsep=1mm,itemindent=1em,listparindent=1em}
\setlist[enumerate,4]{font=\textup, leftmargin=5mm,labelsep=1mm,topsep=1mm,itemsep=1mm,itemindent=1em,listparindent=1em}

\newcommand{\fl}{\mathfrak{l}}

\newcommand{\fp}{\mathfrak{p}}

\newcommand{\fu}{\mathfrak{u}}

\begin{document}
\raggedbottom
\title[Minimal highest weight modules]{Associated varieties of integral minimal  highest weight modules}

\author{Zhanqiang Bai}
\address[Bai]{School of Mathematical Sciences, Soochow University, Suzhou 215006, P. R. China}
\email{zqbai@suda.edu.cn}

\author{Jing Jiang*}
\thanks{* Corresponding author}
\address[Jiang]{
School of Mathematical Sciences, East China Normal University, Shanghai 200241, China}
\email{jjsd6514@163.com}

\author{Rui Wang}
\address[Wang]{School of Mathematical Sciences, Soochow University, Suzhou 215006, P. R. China}
\email{19962879291@163.com}



\bigskip

\begin{abstract}
Let $\mathfrak{g}$ be a complex simple Lie algebra and $L(\lambda)$ be a simple highest weight  module of $\mathfrak{g}$ with highest weight $\lambda-\rho$, where $\rho$ is half the sum of positive roots.    A simple  $\fgg$-module $L_w:=L(-w\rho)$ is called integral minimal if the associated variety of its annihilator ideal equals the closure of the unique minimal special nilpotent orbit. In this paper, we  find that the  associated variety of  any integral minimal module $L_w$ is irreducible and equal to the orbital variety corresponding to a minimal-length element in the Kazhdan--Lusztig right cell containing $w$.

\end{abstract}

\maketitle

\section{Introduction}

Let $G$ be a connected complex simple algebraic group with Lie algebra
$\mathfrak{g}$. For any finitely generated $U(\mathfrak{g})$-module $M$, Bernstein
\cite{Be} defined the associated variety $V(M)$ in $\mathfrak{g}^*$.
Its dimension is the Gelfand--Kirillov dimension of $M$ \cite{Vo78}.
The associated variety is a useful invariant in the study of
$U(\mathfrak{g})$-modules, particularly highest weight modules; see, for
example, \cite{BoB3,Ta,Mel93,Mc00,Wi}. In general, however, its computation
for an irreducible module is difficult.

Fix a Cartan subalgebra $\mathfrak{h}$ and a positive root
system $\Phi^+\subseteq\Phi=\Phi(\mathfrak{g},\mathfrak{h})$, with simple
roots $\Pi$. Let $B$ be the Borel subgroup corresponding to
$\Phi^+$, and let $H\subset B$ be the maximal torus with Lie algebra
$\mathfrak{h}$. Write
$\rho=\frac12\sum_{\alpha\in\Phi^+}\alpha.$  We have  a triangular decomposition $\mathfrak{g}=\mathfrak{n}\oplus \mathfrak{h} \oplus \mathfrak{n}^-$. 
Let $\fp$ be a  parabolic subalgebra with Levi decomposition $\fp=\fl\oplus\fu$.
For a \(\Phi^{+}(\mathfrak{l})\)-dominant integral weight \(\lambda \in \mathfrak{h}^{*}\), let \(F(\lambda)\) be the irreducible \(\mathfrak{l}\)-module with highest weight \(\lambda-\rho\).
 Viewing \(F(\lambda)\) as a \(\mathfrak{p}\)-module (with the nilradical acting trivially), we define the generalized Verma module by
\[
N(\lambda) := U(\mathfrak{g}) \otimes_{U(\mathfrak{p})} F(\lambda).
\]
Let \(L(\lambda)\) denote the irreducible quotient of \(N(\lambda)\), which is a highest weight module of $\mathfrak{g}$ with highest
weight \(\lambda-\rho\).

Denote
by $W$ the Weyl group of $\Phi$. Let $L_w$ denote the simple highest weight $\mathfrak{g}$-module with highest weight $-w\rho-\rho$, where $w\in W$.
Joseph \cite{Jo84} found that the associated variety $V(L_w)$ is a union of orbital varieties, which are defined as follows. For a nilpotent $G$-orbit $\mathcal{O}\subseteq \mathfrak{g}$, the irreducible components of $\overline{\mathcal{O}}\cap \mathfrak{n}$ are called the {\it orbital varieties} of $\mathcal{O}$. Each such component can be expressed as $\mathcal{V}(w)=\overline{B(\mathfrak{n}\cap w(\mathfrak{n}))}$ for some $w\in W$. The associated variety $V(L_w)$ is said to be {\it irreducible} if it consists of a single orbital variety. For highest weight Harish-Chandra modules $L(\lambda)$, the associated variety $V(L(\lambda))$ has been characterized in Bai--Xiao--Xie \cite{BXX} and Bai--Hunziker--Xie--Zierau \cite{BHXZ}.  In general, however, the structure of $V(L_w)$ or $V(L(\lambda))$ remains unknown; see, for example, \cite{BoB3,Ta,Wi}.


As noted in \cite{BoB3}, the associated variety $V(L_w)$ is intimately connected to Schubert varieties. In the following, we recall some of these relations.

 The  flag variety $G/B$ decomposes into a disjoint union of Schubert
cells $BwB/B$, where $w$  is an element in the  corresponding Weyl group $W$. The Schubert
varieties $X(w)=\overline{BwB/B}$ are the closures of the Schubert cells.
Schubert varieties are fundamental objects in algebraic geometry, representation theory, and combinatorics.  $X(w)$ is said to be  smooth if it is a smooth algebraic variety at every point, and  rationally smooth if the Poincar\'{e} polynomial $p_w(t)=\sum_{v\leq w}t^{\ell(v)}$is palindromic, where $\leq$ is the Bruhat order on $W$ and ${\ell}(v)$  the length of $v$.

Investigations into the smoothness and rational smoothness of Schubert varieties have been a major focus since the seminal work of Lakshmibai and Sandhya \cite{LS-90}, who provided a combinatorial characterization for smoothness in type \(A\): a Schubert variety \(X(w) \subset SL(n)/B\) is smooth precisely when the permutation \(w\in S_n\) avoids the patterns $3412$ and $4231$. This result was subsequently generalized to smoothness and rational smoothness in other classical types by Billey \cite{Bi-98}. Then Billey and Postnikov \cite{BP-05} developed a unified method based on root subsystems and embeddings. They demonstrated that, for any semisimple \(G\), the (rational) smoothness of \(X(w) \subset G/B\) is determined by pattern avoidance with respect to a finite list of patterns originating from flat stellar root subsystems (e.g., \(B_2,G_2,A_3, B_3,C_3, D_4\)). From now on, the corresponding Weyl group  element $w$ for a smooth Schubert variety $X(w)$ is called a smooth element.

We refer to \cite{KL79} or \S \ref{KL-cells} for the definition of Kazhdan--Lusztig right (resp. left and two-sided)  cell equivalence relation and use ${\sim}_R$ (resp. ${\sim}_L$ and ${\sim}_{LR}$)  to denote the right  (resp. left and two-sided) cell equivalence relation. We use $\mathcal{C}_R(w)$ (resp. $\mathcal{C}_L(w)$ and $\mathcal{C}_{LR}(w)$) to denote a KL right (resp. left and two-sided) cell  which contains $w$.

From \cite{Jo84}, we know that  the associated variety $V (L_w)$ is constant on each KL right cell.
It is also known that  the associated variety $V (L_w)=\mathcal{V}(w)$ if the Schubert variety $X(w)$ is smooth, see for example \cite[Cor. 4.3.2]{BoB3}.
Therefore, the associated variety $V (L_w)$ will be irreducible if the KL right  cell contains a smooth element. For type $A$, Bai and Chen \cite{Bchen} had given an efficient algorithm to find out all the smooth
elements in a given KL right cell.

Let $\mathcal{C}$ be the two-sided cell of $W$ consisting of the elements
with a unique reduced expression. Our first result determines the smooth
elements in each KL right cell $w_0\mathcal{C}_i$, where $w_0$ is the
longest element of $W$ and $\mathcal{C}_i$ is the right cell of $\mathcal{C}$ containing
$s_i=s_{\alpha_i}$.

\begin{Thm}\label{smcell}
    Let $\mathcal{C}_i$ be the KL  right cell containing the simple reflection $s_i=s_{\alpha_i}$. Then the smooth elements $S(w_0\mathcal{C}_i)$ of  the KL right cell $w_0\mathcal{C}_i$ are given as follows:
    \begin{enumerate}
        \item For type $A_{n}$, we have
$S(w_0\mathcal{C}_{1})=w_0\mathcal{C}_{1}$, $S(w_0\mathcal{C}_{n})=w_0\mathcal{C}_{n}$ and   $$S(w_0\mathcal{C}_{i})=\{w_0s_is_{i-1}\cdots s_{1},w_0s_is_{i+1}\cdots s_{n}\} \text{~for~} 2\leq i\leq n-1.$$
        \item For type $B_n$ with $n\geq 4$, we have $$S(w_0\mathcal{C}_1)=\{w_0s_1\cdots s_ns_{n-1}\cdots s_{n-j+1}\mid 1\leq j\leq n\},$$ $$S(w_0\mathcal{C}_n)=\{w_0s_ns_{n-1}\cdots s_1\}$$ and
    $$S(w_0\mathcal{C}_i)=\{w_0s_is_{i+1}\cdots s_ns_{n-1}\cdots s_1\}\text{~for~}2\leq i\leq n-1.$$
        \item For type $C_n$ with $n\geq 3$, we have $S(w_0\mathcal{C}_1)=w_0\mathcal{C}_1$, $$S(w_0\mathcal{C}_n)=\{w_0s_ns_{n-1}\cdots s_1\}$$ and
    $$S(w_0\mathcal{C}_i)=\{w_0s_is_{i-1}\cdots s_1, w_0s_is_{i+1}\cdots s_ns_{n-1}\cdots s_{1}\}\text{~for~}2\leq i\leq n-1.$$
        \item  For type $D_n$ with $n\geq 4$, we have
$$S(w_0\mathcal{C}_1)=\{w_0s_1\cdots s_{n-2}s_{n-1},\; w_0s_1\cdots s_{n-2}s_{n}\}$$
and $$S(w_0\mathcal{C}_{i})=\varnothing\text{~for~} 2\leq i\leq n-2.$$
For $n=4$, we have $$S(w_0\mathcal{C}_{3})=\{w_0s_{3}s_{2}s_1,w_0s_{3}s_{2}s_4\}$$ and  $$S(w_0\mathcal{C}_4)=\{w_0s_{4}s_{2} s_1,w_0s_{4}s_{2} s_3\}.$$
Moreover, for $n\geq 5$, we have $$S(w_0\mathcal{C}_{n-1})=\{w_0s_{n-1}s_{n-2}\cdots s_1\},$$ $$S(w_0\mathcal{C}_n)=\{w_0s_{n}s_{n-2}\cdots s_1\}.$$
        \item For type $E_r$ with $r=6,7,8$, we have
$$S(w_0\mathcal{C}_{i})=\varnothing\text{~for~} 1\leq i\leq r.$$
        \item For type $F_4$, we have $$S(w_0\mathcal{C}_{i})=\varnothing\text{~for~} 1\leq i\leq 4.$$
        \item For type $G_2$, we have
$$S(w_0\mathcal{C}_{1})=\{s_{2},s_{2}s_1, s_2s_1s_2\}$$ and
$$S(w_0\mathcal{C}_{2})=\{s_{1},s_{1}s_2\}.$$
    \end{enumerate}
\end{Thm}

To determine the smoothness of $X(w)$ for $w\in w_0\mathcal{C}$, we have implemented an algorithm in Python, available at
\begin{center}
\url{https://github.com/JingJiang-web/smooth-Schubert-variety}.
\end{center}

For a highest weight $\mathfrak{g}$-module $L(\lambda)$, the annihilator $\mathrm{Ann}(L(\lambda))$ is a primitive ideal in the enveloping algebra $U(\mathfrak{g})$. The associated graded ideal $\mathrm{gr}(\mathrm{Ann}(L(\lambda)))$ is an ideal inside the symmetric algebra $S(\mathfrak{g})$, which is naturally isomorphic to $\mathrm{gr}(U(\mathfrak{g}))$. The associated variety of  the quotient module $U(\mathfrak{g})/\mathrm{Ann}(L(\lambda))$ is defined as the zero locus of $\mathrm{gr}(\mathrm{Ann}(L(\lambda)))$ in $\mathfrak{g}^*$; this set is called the {\it annihilator variety} of $L(\lambda)$ and is known to be the Zariski closure of a nilpotent orbit $\mathcal{O}_{\mathrm{Ann}(L(\lambda))}$ in $\mathfrak{g}^*$ \cite{Jo85}.  In \cite{BMW} and \cite{BGWX}, the annihilator variety of a highest weight module had been given in a simple way by using the  Robinson--Schensted algorithm and Sommers duality.
But the associated variety of a highest weight module is still a challenging problem.

In the case where $\mathcal{O}_{\mathrm{Ann}(L(\lambda))}$ coincides with $\mathcal{O}_{\min}$, the unique minimal nilpotent orbit of $\mathfrak{g}^*$, the associated variety $V(L(\lambda))$ had been given in a simple and uniform characterization by Bai--Ma--Xiao--Xie \cite{BMXX}. From  \cite{BMXX}, we know that there exists a unique minimal special nilpotent orbit (see \cite[\S 6.3]{CM} for the definition of special orbit), denoted $\mathcal{O}_{\mathrm{ms}}$. It satisfies $\mathcal{O}_{\mathrm{ms}} = \mathcal{O}_{\min}$ for simply-laced Lie algebras, whereas $\mathcal{O}_{\mathrm{ms}} \neq \mathcal{O}_{\min}$ in the non-simply-laced cases. It is known that the minimal special orbit $\mathcal{O}_{{\rm ms}}$ corresponds to the two-sided cell $w_0\mathcal{C}$, see for example \cite{BV82,BMXX,BGWX}. From Kobayashi \cite{Ko11}, we know that small representations of a group will reflect large symmetries in a representation space. This is our motivation for studying small representations of Lie algebras and Lie groups.

By using properties of the associated varieties of highest weight modules and Theorem \ref{smcell}, we can obtain the following result.

\begin{Thm}\label{AV-min}
{Suppose that $V(\Ann(L_w))=\overline{\mathcal{O}_{\mathrm{ms}}}$.}
Then the associated variety $V(L_w)$ is irreducible and can be given as follows:
\begin{enumerate}
   \item {In the non-simply-laced cases, for each $i$ there exists
   a minimal-length element $w_{\rm min}\in w_0\mathcal{C}_i$ such that
   \[
   V(L_w)=\mathcal{V}(w_{\rm min})
   \qquad\text{for all }w\in w_0\mathcal{C}_i.
   \]}

   \item {In the simply-laced cases, for each $i$ there exists
   a minimal-length element $w_{\rm min}\in w_0\mathcal C_i$ such that
   \[
   V(L_w)=\mathcal V(w_{\rm max})=\mathcal V(w_{\rm min})
   \qquad\text{for all }w\in w_0\mathcal C_i,
   \]
   where $w_{\rm max}=w_0s_i$ is the unique maximal-length element
   of the KL right cell $w_0\mathcal C_i$.}
\end{enumerate}

\end{Thm}

A minimal-length element in a right cell need not be unique.
The representatives used in the proofs can be found in the corresponding sections, also see Table \ref{tab:minimal-length-right-cells}.

This paper is organized as follows. In \S \ref{pre}, we  give some necessary preliminaries about the Schubert varieties, associated varieties of highest weight modules and Kazhdan--Lusztig cells. In \S \ref{section-A}, \S \ref{typebc}, \S \ref{typed}, \S \ref{typeef} and \S \ref{typeg2}, we characterize the smooth elements  in the KL right cell $w_0\mathcal{C}_i$ and  the associated variety of $L_w$ in the case of type $A_{n}$, $B_n$ and $C_n$, $D_n$, $E$ and $F_4$, and $G_2$ respectively. In \S \ref{minimal-rep}, we record all the minimal-length elements $w_{\rm min}$ of the right cell $w_0\mathcal C_i$. We use Bourbaki's numbering of the simple roots throughout \cite{Bour}.

\subsection*{Acknowledgments}
Z. Bai  is  supported  by NSFC Grant No. 12671039.

\subsection*{AI disclosure} AI was not involved in the first arXiv version of the paper. In preparing
this revised version, ChatGPT was used primarily to  check  minor mathematical or typographical details. The mathematical ideas, conceptual framework, and proof strategies were developed by the authors, who verified the AI-assisted content and take full responsibility for the manuscript.

\section{Preliminaries}\label{pre}
In this section, we give some brief preliminaries on Schubert varieties, associated varieties of highest weight modules and Kazhdan--Lusztig cells.

\subsection{Schubert variety}

Let $G$ be a complex simple algebraic group and $ B$ be  the standard Borel subgroup of $G$.

\begin{dfn}
The \textit{ Schubert cell} associated to $w \in W$ is:
\[
C(w) = B w B /B\subset G/B
\]
The \textit{Schubert variety} $X(w)$ is its Zariski closure: $X(w) := \overline{C(w)}$.
\end{dfn}


For $v,w\in W$, we write $v\rightarrow w$ if $\ell(w)>\ell(v)$ and $w=vs_\alpha$ for some $\alpha\in \Phi^+$.
Then we define $v<w$ if there exists a sequence $v=w_1\rightarrow w_2\rightarrow\ldots \rightarrow w_m=w$.
The resulting partial order ``$\leq$'' on $W$ is called the {\it Bruhat order\/}. Some more properties about this order can be found in \cite[Chap. 2]{BB05}.

\begin{prop}[{\cite[Thm. 2.2.2]{BB05}}]\label{bruhat-order}
    Let $w=s_1s_2\cdots s_q$ be a reduced expression. Then
    \begin{align*}
        u\leq w \Leftrightarrow \text{there exists a reduced expression } u=s_{i_1}s_{i_2}\cdots s_{i_k}, 1\leq i_1<\cdots<i_k\leq q.
    \end{align*}
\end{prop}


\begin{prop}[{\cite[Cor. 4]{Ca-11}}]\label{inverse}
    A Schubert variety $X(w)$ in $G/B$ is smooth if and only if the corresponding inverse Schubert variety $X(w^{-1})$ is also smooth.
\end{prop}

In simply laced types,  a Schubert variety in $G/B$ is rationally smooth if and only if it is smooth; see \cite{CK03}.





\subsubsection{Pattern avoidance of type $A$}

By the definition, an element $w\in S_n$ is a permutation of the set $\{1,2,...,n\}$. In general, we use $w=(w_1,...,w_n)$ to denote this permutation, where $w_i=w(i)$.

First we have the following definition.

\begin{dfn}
	The element $w = (w_1, ..., w_n)\in S_n$ {\it contains the pattern} $3412$
	 (resp. $4231$) if there exist integers $1\leq i < j < k < l\leq n$ such that
	$w_k < w_l < w_i < w_j $ (resp. $ w_l < w_j < w_k < w_i$).
	 If there are no such integers, we say $w$ {\it avoids the pattern } $3412$ and $4231$.
\end{dfn}

We have the following criterion for smoothness of Schubert varieties.

\begin{prop}[\cite{LS-90}]\label{3412}
	For $\mathfrak{g}=\mathfrak{sl}(n, \mathbb{C})$ and $ W=S_n$,  the Schubert variety $X(w)=\overline{BwB/B}$ is smooth if and only if
	$ w$ avoids the two patterns $3412$ and $4231$.
\end{prop}

In general, $w$ is called a smooth element when $X(w)$ is smooth.

Now we recall the famous Robinson--Schensted insertion  algorithm. Some details can be found in  \cite{Sagan}.

\begin{dfn}[Robinson--Schensted insertion algorithm]
For an element  $ w\in\ S_n $, we write  $w=(w_1,...,w_n)$. We associate to $w $ a  Young tableau  $ P(w) $ as follows. Let $ P_0 $ be an empty Young tableau. Assume that we have constructed Young tableau $ P_k $ associated to $ (w_1,\cdots,w_k) $, $ 0\leq k<n $. Then $ P_{k+1} $ is obtained by adding $ w_{k+1} $ to $ P_k $ as follows. Firstly we add $ w_{k+1} $ to the first row of $ P_k $ by replacing the leftmost entry $ x $ in the first row which is \textit{strictly} bigger than $ w_{k+1} $.  (If there is no such an entry $ x $, we just add a box with entry $w_{k+1}  $ to the right side of the first row, and end this process). Then add $ x $ to the next row as the same way of adding $w_{k+1} $ to the first row.  Finally we put $P(w)=P_n$.
Let $Q(w)$ be the recording tableau such that $1,2,...,n$ are placed in the $Q$'s so that shape of $P_k$ equal to the shape of $Q_k$ for all $1\leq k\leq n$. Thus $Q(w)=Q_n$ and $\mathrm{sh}(P(w))=\mathrm{sh}(Q(w))$.
\end{dfn}

Recall that  a {partition} of $n$ is a decreasing sequence $p=[p_1, p_2, \cdots, p_k]$ of positive integers such that $p_1+p_2+\cdots+p_k=n$ for some $1\leq k\leq n$.
For any $w\in S_n$, we denote $p(w) = \mathrm{sh}(P(w))=[p_1,p_2,\dots,p_k]$, where $p_i$ is the number of boxes in the $i$-th row of $P(w)$. Let $\mathrm{Seq}_n (\mathbb{Z})$ denote the set of integer sequences of length $n$. For $x\in \mathrm{Seq}_n (\mathbb{Z})$, we still can apply the RS algorithm to get a Young tableau $P(x)$. See \cite{BX} for more details.



\subsubsection{The root subsystem pattern avoidance criterion}
We recall the Billey--Postnikov criterion for smoothness of Schubert varieties in terms of patterns in root subsystems \cite{BP-05}.

Following Billey--Postnikov \cite[\S2]{BP-05}, let
$E=\operatorname{span}_{\mathbb R}\Phi$.
A \emph{flat root subsystem} of $\Phi$ is a set
$\Delta=\Phi\cap U$ for a real linear subspace $U\subseteq E$.
Equivalently,
$\Delta=\Phi\cap\operatorname{span}_{\mathbb R}\Delta$.
We put $\Delta^+=\Delta\cap\Phi^+$.


For $w\in W$, define
\[
    \Phi_w:=\Phi^+\cap w\Phi^-=\{\alpha \in \Phi^+\mid w^{-1}\alpha  \in \Phi^-\}.
\]
 Then we have $|\Phi_w|=\ell(w)$, the length of $w$.





By \cite[Lem. 2.1]{BP-05}, for each $w\in W(\Phi)$ and any flat root subsystem $\Delta\subseteq \Phi$, one has $\Phi_w\cap \Delta=\Delta_{\sigma}$ for a single element $\sigma\in W(\Delta)$.
We define the {\it flattening map} $f_\Delta:W(\Phi)\to W(\Delta)$ by taking $f_\Delta(w)=\sigma$ where $\sigma$ is determined by the condition $\Delta_{\sigma} = \Phi_w\cap \Delta$.

A graph is called a {\it star\/} if it is connected and contains a vertex incident with all edges.  We say that a root system $\Delta$ is {\it stellar\/} if its Dynkin diagram is a star and $\Delta$ is not of type $A_1$ or $A_2$.

In type $F_4$, the flat stellar subsystems have types $B_2$, $B_3$,
and $C_3$. To see this, recall that a flat subsystem is Weyl-conjugate
to a standard parabolic root subsystem: choose a vector orthogonal
precisely to its roots and move it into the closed fundamental chamber.
The roots orthogonal to this vector are generated by the simple roots
orthogonal to it. The assertion now follows from the $F_4$ Dynkin diagram.
In particular, an $A_3$ reflection closure in $F_4$ is not flat.
For types $E_6,E_7,E_8$, the stellar subsystems in the BP list
have types $A_3$ and $D_4$.

\begin{Thm}[{\cite[Thm 2.2]{BP-05}}]\label{BP-main}
Let $G$ be a semisimple simply connected Lie group,
$B$ be any Borel subgroup, with
corresponding root system $\Phi$ and Weyl group $W=W(\Phi)$.  For
$w\in W$, the Schubert variety $X(w) \subseteq  G/B$ is smooth (rationally
smooth) if and only if, for every flat stellar root subsystem $\Delta$ in
$\Phi$, the pair $(\Delta^+,f_\Delta(w))$ is smooth (rationally
smooth).
\end{Thm}

If $\Delta$ is a flat root subsystem of $\Phi$ and $\sigma = f_\Delta(w)$, then we say that the element $w$ in $W(\Phi)$ \textit{contains the pattern} $(\Delta^+,\sigma)$.

We say that an element $w$ \textit{avoids the pattern} $(\Delta^+,\sigma)$ if $w$ does not contain a pattern isomorphic to $(\Delta^+,\sigma)$.

We write $[a,b,\dots,c]$ to denote the collection of words $1,a,b,\dots,c$ and use $[a,b,c]\,d$ to represent the four words $d, ad, bd, cd$.

\begin{Thm}[{\cite[Thm. 2.4]{BP-05}}] \label{BP-patterns}
Let $\Phi$ be an arbitrary root system.  The Schubert variety
$X(w)$, $w\in W(\Phi)$, is smooth if and only if $w$ avoids the patterns
listed in the following table in every flat root subsystem of $\Phi$:
\smallskip
\begin{center}
\begin{tabular}{c|c|c}
\hline
\text{stellar type}   &  \text{forbidden patterns} & \# \text{patterns} \\
\hline
\mystrut
$B_2$  & $s_2s_1s_2$ & $1$ \\[.05in]
\hline
\mystrut
$G_2$ & $[s_2]s_1s_2s_1[s_2]$,
      $s_1s_2s_1s_2s_1$ & $5$
   \\[.05in]
\hline
\mystrut
$A_3$ & $s_2 s_1 s_3 s_2$, $s_1 s_2 s_3 s_2 s_1$ & $2$ \\[.05in]
\hline
\mystrut
$B_3$ &
$s_2 s_1 s_3 s_2$,
$s_1 s_2 s_3 s_2 s_1[s_3, s_3 s_2, s_2 s_3, s_2 s_3 s_2]$ & $6$
\\[.05in]
\hline
\mystrut
$C_3$ &
$[s_3]s_2 s_1 s_3 s_2[s_3]$,
$s_3 s_2 s_1 s_2 s_3$,
$s_1 s_2 s_3 s_2 s_1  s_3 s_2 s_3$
& $6$
\\[.05in]
\hline
\mystrut
$D_4$ &  $s_2 s_1 s_3 s_4 s_2$ & $1$ \\[.05in]
\hline
\end{tabular}
\end{center}
\medskip
\label{th:smooth-general}
\end{Thm}



We use Theorem \ref{BP-patterns} for exceptional types and Billey's explicit pattern lists \cite{Bi-98} for the classical types.
\begin{dfn}\label{signed-notation}
Given non-zero real numbers $a_1,\ldots,a_k$ such that
$|a_i|\ne|a_j|$ whenever $i\ne j$, denote $-a$ by $\bar{a}$ and define ${\rm fl}(a_1, a_{2}, \cdots, a_{k})$ to be the unique element
$b=(b_1, \cdots, b_k)$ in $W({B}_k)$ such that
\begin{itemize}
\item
both $a_j$ and $b_j$ have the same sign.
\item for all $i,j$, we have $|b_i|<|b_j|$ if and only if $|a_i|<|a_j|$.
\end{itemize}
\end{dfn}
For example, ${\rm fl}(\bar{5}, 4, \bar{6}, 2)=\bar{3}2 \bar{4} 1$.  Any
sequence containing the subsequence $\bar{6}, 3, \bar{7}, 1$ does not
avoid the pattern $\bar{3} 2 \bar{4} 1$.

\begin{Thm}[{\cite[Thm. 7.2]{BP-05}}]\label{t:classical.patterns}
Let $W$ be one of the Weyl groups $W(A_{n-1})$, $W(B_{n})$, $W(C_{n})$ or
$W(D_{n})$ with $n\geq4$ and let $w \in W$.
Then $X(w)$ is (rationally) smooth if and
only if for each subsequence $1\leq i_1<i_2<i_3 <i_4 \leq n$,
${\rm fl}(w_{i_1},w_{i_2},w_{i_3},w_{i_4})$ corresponds to a (rationally)
smooth Schubert variety.
\end{Thm}

\subsection{Associated varieties}\label{associated-varieties}
{
We fix the Killing form $\kappa$ on $\mathfrak{g}$. Since $\mathfrak{g}$ is simple, this form is nondegenerate and $G$-invariant, and hence induces a $G$-equivariant identification
\[
\mathfrak{g}\longrightarrow\mathfrak{g}^*,\quad X\longmapsto\kappa(X,\,\cdot\,).
\]
Throughout the remainder of the paper, we use this identification without further comment. Thus coadjoint nilpotent orbits and associated varieties, which are initially defined in $\mathfrak{g}^*$, will also be regarded as subsets of $\mathfrak{g}$. In particular, the associated varieties, orbital varieties, and parabolic nilradicals appearing in the inclusions below are viewed in the same  space.
}

Let $\mathfrak{g}$ be a simple complex Lie algebra, and let $M$ be a
finitely generated $U(\mathfrak{g})$-module. Choose a finite-dimensional
generating subspace $M_0$ of $M$. Let $U_{n}(\mathfrak{g})$ be the standard filtration of $U(\mathfrak{g})$. Set $M_n=U_n(\mathfrak{g})\cdot M_0$ and
\(
\text{gr} (M)=\bigoplus\limits_{n=0}^{\infty} \text{gr}_n M,
\)
where $M_{-1}=0$ and $\text{gr}_n M=M_n/M_{n-1}$. Thus $\text{gr}(M)$ is a graded module of $\text{gr}(U(\mathfrak{g}))\simeq S(\mathfrak{g})$.



\begin{dfn}
	The  \textit{associated variety} of $M$ is defined by
	\begin{equation*}
	V(M):=\{X\in \mathfrak{g}^* \mid f(X)=0 \text{ for all~} f\in \operatorname{Ann}_{S(\mathfrak{g})}(\operatorname{gr} M)\}.
	\end{equation*}
\end{dfn}

This definition is independent of the choice of $M_0$; see, for
example, \cite{NOT}.

\begin{dfn} Let $\mathfrak{g}$ be a finite-dimensional semisimple Lie algebra. Let $I$ be a two-sided ideal in $U(\mathfrak{g})$. Then $\text{gr}(U(\mathfrak{g})/I)\simeq S(\mathfrak{g})/\text{gr}I$ is a graded $S(\mathfrak{g})$-module. Its annihilator is $\text{gr}I$. We define its associated variety by
	$$V(I):=V(U(\mathfrak{g})/I)=\{X\in \mathfrak{g}^* \mid p(X)=0\ \mbox{for all $p\in {\text{gr}}I$}\}.
	$$
\end{dfn}

Recall that  $L_w$ with $ w\in W$  is a simple highest weight $\mathfrak{g}$-module of highest weight $-w\rho-\rho$.  We denote $I_w=\Ann(L_w)$.  Borho--Brylinski \cite{BoB1} proved that the associated variety of  $I_w$ is  the closure of a single nilpotent orbit $\mathcal{O}_w$. Thus, we also write  $V(I_w):= V(\Ann (L_w))=\overline{\mathcal{O}_w}$.

 Furthermore, we have the following proposition.

\begin{prop}[\cite{Jo85}]
	Let $\mathfrak{g}$ be a reductive Lie algebra and $I$ be a primitive ideal in $U(\mathfrak{g})$.Then $V(I)$ is the closure of a single nilpotent coadjoint orbit $\mathcal{O}_I$ in $\mathfrak{g}^*$. In particular, for a highest weight module $L(\lambda)$, we have $V(\mathrm{Ann} (L(\lambda)))=\overline{\mathcal{O}_{\mathrm{Ann}(L(\lambda))}}$.
\end{prop}

Let $\mathfrak{b}=\mathfrak{h}\oplus\mathfrak{n}$ be the Borel
subalgebra corresponding to the fixed positive system $\Phi^+$, and let
$\mathfrak{n}^-$ be the opposite nilradical. Thus
$\mathfrak{g}=\mathfrak{n}\oplus\mathfrak{h}\oplus\mathfrak{n}^-$.
{For a nilpotent $G$-orbit $\mathcal{O}$, the   irreducible components of $\overline{\mathcal{O}}\cap \mathfrak{n}$  are called {\it orbital varieties} associated to $\mathcal{O}$.  {After Steinberg \cite{St76}}, an orbital variety has the following form
$$\mathcal{V}(w)=\overline{B(\mathfrak{n}\cap{w(\mathfrak{n})} )}$$ for some  $w$ in the Weyl group $ W=N_G(H)/H$, where $B$ is the Borel subgroup corresponding to the
Borel subalgebra $\mathfrak{b}$ of $\mathfrak{g}$.
Here $w(\mathfrak n)=\operatorname{Ad}(\dot w)\mathfrak n$,
where $\dot w\in N_G(H)$ represents $w$; this subspace is independent
of the representative because $H$ normalizes $\mathfrak n$.}

\begin{prop}[{\cite[II. 9.8]{Sp82}}]\label{geometriccell}
    Suppose  $x,y\in S_n$. Then we have $\mathcal{V}(x)=\mathcal{V}(y)$ if and only if $x\sim_{R} y$.
\end{prop}

This equivalence does not hold in general outside type $A$.

We also have the following results.

\begin{prop}[\cite{Jo84}]
	Let $L(\lambda)$ be a highest weight module of a simple Lie algebra $\mathfrak{g}$ with highest weight $\lambda-\rho$. Then its associated variety $V(L(\lambda))$ equals the union of some orbital varieties associated with the nilpotent coadjoint orbit $\mathcal{O}_{\Ann(L(\lambda))}$ in $\mathfrak{g}^*$.

\end{prop}

The associated variety $V(L(\lambda))$ is  {irreducible}  {if and only if it coincides with a single orbital variety}.

\begin{prop}[{\cite[Cor. 4.3.2]{BoB3}}]\label{smoothequal}
	If the Schubert variety $X(w)=\overline{BwB/B}$ is smooth, then
$V(L_w)=\mathcal{V}(w)$.

\end{prop}

{
For a subset ${\rm J}\subseteq\Pi$, set
\[
\Phi_{\rm J}:=\Phi\cap\mathbb{Z}{\rm J},\qquad
\Phi_{\rm J}^+:=\Phi_{\rm J}\cap\Phi^+,\qquad
\mathfrak{u}_{\rm J}:=
\bigoplus_{\alpha\in\Phi^+\setminus\Phi_{\rm J}^+}\mathfrak{g}_\alpha.
\]
Thus $\mathfrak{u}_{\rm J}$ is the nilradical of the standard parabolic
subalgebra corresponding to ${\rm J}$. We say that
$\lambda\in\mathfrak{h}^*$ is \emph{$\Phi_{\rm J}^+$-dominant} if
\[
{\langle\lambda,\alpha^\vee\rangle}\in\mathbb{Z}_{>0}
\quad\text{for every }\alpha\in{\rm J}.
\]

}

\begin{prop}[{\cite[Prop. 4.1]{BMXX}}]\label{Same}
    Let $\frg$ be a simple complex Lie algebra and $\lambda\in\frh^*$. Then $\lambda$ is $\Phi_{\rm J}^+$-dominant if and only if $V(L(\lambda))\subseteq \fru_{\rm J}$.
\end{prop}

\subsection{Kazhdan--Lusztig cells}\label{KL-cells}
Now we  recall the Kazhdan--Lusztig  cells in the following. Some more details can be found in \cite{Lus03}.

Recall that the Weyl group $ W  $ is a Coxeter group generated by $ S=\{s_\al\mid\al\in\Pi \} $. Given an indeterminate $q$, the Hecke algebra $ \mc{H} $ over $ \mathcal{A} :=\mathbb{Z}[q,q^{-1}]$ is generated by $ T_w $, $ w\in W $ with relations \[
T_wT_{w'}=T_{ww'} \text{ if }\ell(ww')=\ell(w)+\ell(w'),
\]
\[
\text{and }(T_s+q^{-1})(T_s-q)=0 \text{ for any }s\in S.
\]
The unique elements $ C_w $ such that
\[
\overline{C_w}=C_w,\qquad C_w\equiv T_w \mod{\mc{H}_{<0}}
\]
are known as the \textit{Kazhdan--Lusztig} (KL) \textit{basis} of $ \mc{H} $, where $ \bar{\,} :\mc{H}\rar\mc{H}$ is the bar involution such that $ \bar{q}=q^{-1} $, $ \overline{T_w} =T_{w^{-1}}^{-1}$, and $ \mc{H}_{<0}=\bigoplus_{w\in W}\mathcal{A}_{<0}T_w $ with $ \mathcal{A}_{<0}=q^{-1}\mathbb{Z}[q^{-1}] $.

If $ C_y $ occurs in the expansion of $ hC_w $ (resp. $C_wh$) with respect to the KL basis for some $ h\in\mc{H} $, then we write $ y\leftarrow_L w $ (resp. $ y\leftarrow_R w $). Extend $ \leftarrow_L $ (resp. $ \leftarrow_R $) to a preorder $ \leq_L $ (resp. $\leq _R$) on $ W $. For $x, w\in W$, write $x \leq_{LR} w$ if there exists $x=w_1, \cdots, w_n=w$ such that for every $1\leq i<n$ we have either $w_i\leq_L w_{i+1}$ or $w_i\leq_R w_{i+1}$. Let $\sim_{L}$, $\sim_{R}$, $\sim_{LR}$ be the equivalence relations associated with $\leq_L$, $\leq_R$, $\leq_{LR}$ (for example, $x\sim_{L}w$ if and only if $x\leq_L w$ and $w\leq_Lx$). The equivalence classes on $W$ for $\sim_L$, $\sim_R$, $\sim_{LR}$ are called \textit{left cells}, \textit{right cells} and \textit{two-sided cells} respectively.

\begin{lem}[{\cite[Lem. 6.6]{Jo84}}; {\cite[Cor. 6.3]{BoB3}}]\label{constant}
$V(L_w)$ is constant on each KL right cell.
\end{lem}

\begin{lem}[{\cite[Thm. 2.1]{Mel93}}]\label{gamma-w}
    For each $w\in W$, there exists a subset $\Gamma(w)\subseteq  W$ such that
    $$V(L_w)=\bigcup\limits_{y\in \Gamma(w)}\mathcal{V}(y)$$
    where $\Gamma(w)$ has the following properties:
    \begin{enumerate}
        \item $w\in \Gamma(w)$.
        \item $\Gamma(w)\subseteq  \mathcal{C}_{LR}(w)$.
        \item If $y\in \Gamma(w)$, then $y\leq w$.
    \end{enumerate}
\end{lem}
\begin{lem}[{\cite[Prop. 4.1]{Mel93}}]\label{min-length-w}
In each two-sided cell $\mathcal{C}_{LR}$, there exists at least one right cell $\mathcal{C}_R$ such that $V(L_w)$ is irreducible for all $w \in \mathcal{C}_R$. Moreover, let $w \in \mathcal{C}_{LR}$ be an element of minimal length in $\mathcal{C}_{LR}$, i.e., $\ell(w) \le \ell(y)$ for every $y \in \mathcal{C}_{LR}$. Then
\[
V(L_y) = \mathcal{V}(w) \quad \text{for all } y \in \mathcal{C}_R(w).
\]
\end{lem}

The map $w\mapsto \mathcal{O}_w$ induces a bijection between the two-sided cells in the Weyl group $W$ and special nilpotent orbits, see for example \cite{BV82, BoB1, Ta}. In other words, $\mathcal{O}_w=\mathcal{O}_y$ if and only if $w\sim_{LR}y$. A nilpotent orbit of $\mathfrak{g}$ or $\mathfrak{g}^*$ is called {\it minimal} if it is minimal among non-trivial nilpotent orbits with respect to the order defined by $\mathcal{O}' \leq \mathcal{O} \Leftrightarrow \mathcal{O}' \subseteq  \overline{\mathcal{O}}$. Similarly,
a special nilpotent orbit $\mathcal{O}_w$ is called  {\it minimal special}  if it is minimal among non-trivial special nilpotent orbits. From Kim \cite{Kim19}, we know that the minimal and  minimal special nilpotent orbit coincide if and only if $\mathfrak{g}$ is of simply-laced type. Also these orbits are unique by \cite[\S 2.8]{Kim19}. From Wang \cite{Wang}, the minimal nilpotent orbit $\mathcal{O}_{\min}$ has the minimal dimension $2{\langle\rho,\beta^{\vee}\rangle}$ among all nilpotent orbits, where $\beta$ is the  highest root.  From Humphreys \cite[Prop. 5]{Hum16}, the minimal special nilpotent orbit $\mathcal{O}_{\mathrm{ms}}$ has the minimal dimension $2{\langle\rho,\beta_s^{\vee}\rangle}$ among all special nilpotent orbits, where $\beta_s$ is the highest short root  when $\Phi$ is not simply laced and is the highest root $\beta$ otherwise.


Let
\[
\cC := \Set{w\in W | \text{$w\neq 1$ and has a unique   reduced expression}}.
\]
Evidently, $S\subseteq \caC$.
Denote \[
\mathcal{C}w_0:=\{ww_0\mid w\in\mathcal{C}\}.
\]

\begin{lem}[{\cite[Lem. 3.3]{BMXX}}]\label{hklem2}
Let $w\in W$. Then
	\begin{itemize}
 \item[(1)]  The sets $\mathcal{C}$ and $\mathcal{C}w_0$ are two-sided cells of $W$.
		\item[(2)]  If $w\neq 1, w_0$, then $x\leq_{LR} w\leq_{LR} y$ for any $x\in\mathcal{C}w_0$ and $y\in\mathcal{C}$.
            \item[(3)] The equations $\cC = w_0^{-1} \cC w_0$ and $\cC w_0 = w_0\cC$ hold.
	\end{itemize}

\end{lem}





From \cite[Rem. 3.6]{BMXX}, $L_w$  is called an {\it integral  minimal} module when    $w\in \mathcal{C}w_0$. In this case, we have $V(I_w)=\overline{\mathcal{O}_w}=\overline{\mathcal{O}_{\mathrm{ms}}}$.
When $V(\Ann(L(\lambda)))=\overline{\mathcal{O}_{\Ann(L(\lambda))}}
=\overline{\mathcal{O}_{\min}}$, we call $L(\lambda)$ a
\emph{minimal} highest weight module.

\begin{lem}[{\cite[Thm. 4.2]{BMXX}}]\label{minimal-integralequal}
 If $L_w$ is a minimal highest weight module, then
\[
V(L_w) = \bigcup_{\alpha \in \Pi \setminus I_\lambda} \overline{Be_\alpha}=\bigcup_{\alpha \in \Pi \setminus I_\lambda}\mathcal{V}(s_\alpha w_0),
\]
where $0\neq e_\alpha\in\mathfrak{g}_\alpha$ and
{
\[
I_\lambda=\{\alpha\in\Pi\mid
{\langle\lambda,\alpha^\vee\rangle}\in\mathbb{Z}_{>0}\}
 \text{~with~}\lambda=-w\rho.
\]
}
Moreover, $I_\lambda$ contains all the simple short roots.
\end{lem}



For each $s\in S$, denote $\mathcal{C}_s=\{w\in \mathcal{C}\mid \mathcal{L}(w)=\{s\}\}$, where $\mathcal{L}(w)=\{s\in S\mid \ell(sw)< \ell(w)\}$.

\begin{lem}[{\cite[\S 3.5]{Lus83}}]
$\mathcal{C}_s$ is a KL right cell for any $s\in S$.

\end{lem}

For each $\alpha_i\in \Pi$, we denote the corresponding simple reflection by $s_i:=s_{\alpha_i}$ and the KL right cell by $\mathcal{C}_i:=\mathcal{C}_{s_i}$. Let $ \preceq $ be the weak Bruhat order which is defined by \[
x\preceq y\text{ iff } y=xz \text{ with }\ell(y)=\ell(x)+\ell(z).
\]

Note that $\mathcal{C}_i  $ is the set of $ w\in \mathcal{C} $ such that $ s_{\alpha_i} w<w $. These elements can be obtained  by examining admissible walks in the Dynkin diagram.
This gives the following explicit descriptions.

\begin{lem}\label{rightcell-ci}
    The KL right cell $\mathcal{C}_i$ can be given as follows:
\begin{enumerate}
    \item For type $A_{n}$, we have
    $$\mathcal{C}_i=\{w\mid s_i\preceq w \preceq s_is_{i+1}\cdots s_{n}\text{ or }s_i\cdots s_1\}.$$

\item For type $B_{n}$ and $C_n$, we have
    \begin{align*}
    \mathcal{C}_i=&\{w\mid s_i\preceq w\preceq s_is_{i-1}\cdots s_1 \text{~or~} s_is_{i+1}\cdots s_{n}s_{n-1}s_{n-2}\cdots s_{2}s_1\}\\
    &\text{~for~} 1\leq i\leq n-1,\\
    \mathcal{C}_n=&\{w\mid s_n\preceq w\preceq s_ns_{n-1}\cdots s_1 \text{~or~} s_ns_{n-1}s_{n}\}.
    \end{align*}

\item For type $D_{n}$, we have
    \begin{align*}
        \mathcal{C}_i=\{ w\mid &s_i\preceq w\preceq s_is_{i+1}\cdots s_{n-2} s_{n-1} \text{~or~} s_is_{i+1}\cdots s_{n-2} s_{n}\\
        &\text{~or~} s_i\cdots s_1\} \text{~for~} 1\leq i\leq n-2,\\
        \mathcal{C}_{n-1}=\{w\mid & s_{n-1}\preceq w\preceq   s_{n-1}s_{n-2}\cdots s_{1}\text{~or~}s_{n-1}s_{n-2}s_n\},\\
        \mathcal{C}_{n}=\{w\mid &s_{n}\preceq w\preceq s_{n}s_{n-2}\cdots s_{1} \text{~or~}s_{n}s_{n-2}s_{n-1}\}.
    \end{align*}
\item For type $E_{6}$, we have
\begin{align*}
\mathcal{C}_1&=\{ w\mid 1\preceq w\preceq 13456 \text{ or }1342 \},\\
\mathcal{C}_2&=\{ w\mid 2\preceq w\preceq 2431 \text{ or }2456 \},\\
\mathcal{C}_3&=\{ w\mid 3\preceq w\preceq 31 \text{ or }3456 \text{ or }342  \},\\
\mathcal{C}_4&=\{ w\mid 4\preceq w\preceq 431 \text{ or }456 \text{ or }42 \},\\
\mathcal{C}_5&=\{ w\mid 5\preceq w\preceq 5431 \text{ or }542 \text{ or }56 \},\\
\mathcal{C}_6&=\{ w\mid 6\preceq w\preceq 65431 \text{ or }6542 \}.
\end{align*}

\item For type $E_{7}$, we have
\begin{align*}
\mathcal{C}_1&=\{ w\mid 1\preceq w\preceq 134567 \text{ or }1342 \},\\
\mathcal{C}_2&=\{ w\mid 2\preceq w\preceq 2431 \text{ or }24567 \},\\
\mathcal{C}_3&=\{ w\mid 3\preceq w\preceq 31 \text{ or }34567 \text{ or }342  \},\\
\mathcal{C}_4&=\{ w\mid 4\preceq w\preceq 431 \text{ or }4567 \text{ or }42 \},\\
\mathcal{C}_5&=\{ w\mid 5\preceq w\preceq 5431 \text{ or }542 \text{ or }567 \},\\
\mathcal{C}_6&=\{ w\mid 6\preceq w\preceq 67 \text{ or } 65431 \text{ or }6542 \},\\
\mathcal{C}_7&=\{ w\mid 7\preceq w\preceq  765431 \text{ or }76542 \}.
\end{align*}

\item For type $E_{8}$, we have
\begin{align*}
\mathcal{C}_1&=\{ w\mid 1\preceq w\preceq 1345678 \text{ or }1342 \},\\
\mathcal{C}_2&=\{ w\mid 2\preceq w\preceq 2431 \text{ or }245678 \},\\
\mathcal{C}_3&=\{ w\mid 3\preceq w\preceq 31 \text{ or }345678 \text{ or }342  \},\\
\mathcal{C}_4&=\{ w\mid 4\preceq w\preceq 431 \text{ or }45678 \text{ or }42 \},\\
\mathcal{C}_5&=\{ w\mid 5\preceq w\preceq 5431 \text{ or }542 \text{ or }5678 \},\\
\mathcal{C}_6&=\{ w\mid 6\preceq w\preceq 678 \text{ or } 65431 \text{ or }6542 \},\\
\mathcal{C}_7&=\{ w\mid 7\preceq w\preceq  765431 \text{ or }76542\text{ or }78 \},\\
\mathcal{C}_8&=\{ w\mid 8\preceq w\preceq  8765431 \text{ or }876542 \}.
\end{align*}
\item For type $F_{4}$, we have
\begin{align*}
\mathcal{C}_1&=\{ w\mid 1\preceq w\preceq 1234 \text{ or }12321 \},\\
\mathcal{C}_2&=\{ w\mid 2\preceq w\preceq 21 \text{ or }234\text{ or }2321 \},\\
\mathcal{C}_3&=\{ w\mid 3\preceq w\preceq 34 \text{ or }321 \text{ or }3234  \},\\
\mathcal{C}_4&=\{ w\mid 4\preceq w\preceq 4321 \text{ or }43234 \}.
\end{align*}

\item For type $G_{2}$, we have
\begin{align*}
\mathcal{C}_1&=\{ w\mid 1\preceq w\preceq 12121  \},\\
\mathcal{C}_2&=\{ w\mid 2\preceq w\preceq 21212 \}.
\end{align*}

\end{enumerate}

\end{lem}

In the above Lemma \ref{rightcell-ci}, for exceptional types we use $``ijk\cdots"$ to represent the Weyl group element $``s_is_js_k\cdots"$.

\begin{proof}
For distinct simple reflections $s_a,s_b$, let $m_{ab}$ be
the order of $s_as_b$. By the word property
\cite[Thm.~3.3.1]{BB05}, a nonempty word is a unique reduced
expression if and only if it contains neither adjacent repeated
letters nor an alternating consecutive subword in $s_a,s_b$
of length $m_{ab}$.
Indeed, if neither occurs, no cancellation or braid move can be
applied. The word property then implies both reducedness and
uniqueness. Conversely, an adjacent repetition contradicts
reducedness, whereas a braid move produces another reduced
expression.

The left descent set of an element with a unique reduced
expression consists of its first letter. Hence $\mathcal C_i$
consists precisely of the elements whose unique reduced
expressions begin with $s_i$.
Moreover, if $v$ has a unique reduced expression, then
$w\preceq v$ if and only if the reduced expression of $w$ is
a prefix of that of $v$. Thus it remains to determine the
possible words beginning with $s_i$.

For types $A_n,D_n,E_r$, consecutive letters must correspond
to adjacent vertices of the Dynkin diagram, since otherwise
they commute. Immediate reversal along an edge is excluded
by the braid relation $s_as_bs_a=s_bs_as_b$.
Since these diagrams are trees, the admissible words correspond
exactly to simple paths beginning at the simple root $\alpha_i$.
In type $A_n$, such a path proceeds towards one of the two
endpoints of the Dynkin diagram. In type $D_n$, it follows the chain and, upon reaching
the branching vertex $\alpha_{n-2}$, may continue along either remaining
branch. In type $E_r$, the maximal paths terminate at the simple roots
$\alpha_1$, $\alpha_2$, or $\alpha_r$.
Reading these paths in the Bourbaki numbering gives precisely
the displayed weak-order intervals for these types.

For types $B_n$ and $C_n$, the only edge along which an
immediate reversal is allowed is the double edge joining
$\alpha_{n-1}$ and $\alpha_{n}$, since $m_{n-1,n}=4$.
A second consecutive reversal along this edge would produce
an alternating subword of length four and is therefore
excluded.
If $i<n$, a word beginning with $s_i$ either proceeds towards
$\alpha_1$, or proceeds towards $\alpha_n$ and may then return towards $\alpha_1$.
Consequently, every such word is a prefix of
\[
  s_is_{i-1}\cdots s_1
  \quad\text{or}\quad
  s_is_{i+1}\cdots s_ns_{n-1}\cdots s_1.
\]
If $i=n$, the word either proceeds towards $\alpha_1$, or follows
$s_ns_{n-1}s_n$ and must stop. This gives the stated
description of $\mathcal C_n$.

For type $F_4$, immediate reversal is possible only along
the double edge joining $\alpha_2$ and $\alpha_3$, and an alternating
subword of length four on this edge is excluded.
Starting at $\alpha_1$, the maximal admissible words are
$s_1s_2s_3s_4$ and $s_1s_2s_3s_2s_1$.
Starting at $\alpha_2$, they are
$s_2s_1$, $s_2s_3s_4$, and $s_2s_3s_2s_1$.
The symmetry $1\leftrightarrow4$, $2\leftrightarrow3$
of the Coxeter diagram gives the corresponding lists
starting at $\alpha_4$ and $\alpha_3$.

Finally, in type $G_2$, every admissible word alternates
between $s_1$ and $s_2$. Since $m_{12}=6$, its length is
at most five, giving the two stated intervals.

Each listed endpoint word, and hence every nonempty prefix,
has no adjacent repeated letters and admits no braid move.
By the word property, each such prefix is therefore a unique
reduced expression. Conversely, the preceding case analysis
shows that every unique reduced expression beginning with
$s_i$ is one of these prefixes. Thus these prefixes give
exactly $\mathcal C_i$. Now the proof is finished. 
\end{proof}



Throughout, $s_i=s_{\alpha_i}$, with the simple roots numbered
as in Bourbaki \cite{Bour}. In type $A_n$, we use the usual
permutation notation $w=(w(1),\ldots,w(n+1))$.

For types $B_n,C_n,D_n$, we use the signed-permutation
realization of \cite[\S5.1, (5.1)]{BXX}.
Let $f_j=\epsilon_{n+1-j}$, where
$\epsilon_1,\ldots,\epsilon_n$ are the standard coordinate vectors.
We write $w=(w_1,\ldots,w_n)$, with
$\{|w_1|,\ldots,|w_n|\}=\{1,\ldots,n\}$, according to
\[
  w(f_j)=\operatorname{sgn}(w_j)f_{|w_j|}
  \qquad(1\leq j\leq n).
\]
In type $D_n$, the number of negative entries is even.

With the convention $(uv)(f_j)=u(v(f_j))$, right multiplication
by $s_i$ for $1\leq i<n$ interchanges the entries in positions
$n-i$ and $n-i+1$, while
\[
  ws_n=
  \begin{cases}
    (-w_1,w_2,\ldots,w_n)
      & \text{in types }B_n,C_n,\\
    (-w_2,-w_1,w_3,\ldots,w_n)
      & \text{in type }D_n.
  \end{cases}
\]
We use this convention in Table~\ref{Elements-in-cell}
and all signed one-line expressions below.

\begingroup
\small
\setlength{\tabcolsep}{4pt}
\renewcommand{\arraystretch}{1.18}
\setlength{\LTleft}{\fill}
\setlength{\LTright}{\fill}

\begin{longtable}{@{}c @{\hspace{0.8cm}} c @{\hspace{1cm}} >{\raggedright\arraybackslash}p{0.5\textwidth}c@{}}
\caption{List of elements in the right cell $\mathcal{C}_i$}
\label{Elements-in-cell}\\

\toprule
Type & Right cell &
Elements in right cells of $W(\Phi)$
& $\#\mathcal{C}_i$ \\
\midrule
\endfirsthead

\multicolumn{4}{l}{\tablename~\thetable\ (continued)}\\
\toprule
Type & Right cell &
Elements in right cells of $W(\Phi)$
& $\#\mathcal{C}_i$ \\
\midrule
\endhead

\midrule
\multicolumn{4}{r}{\footnotesize Continued on next page}\\
\endfoot

\bottomrule
\endlastfoot


\multirow[c]{6}{*}{$A_n$} & \multirow[c]{6}{*}{$\mathcal{C}_i(1\leq i\leq n)$}
& $(1,2,\cdots,i-1,i+1,i,\cdots,n+1)$
& \multirow[c]{6}{*}{$n$} \\

& & \qquad $\vdots$ & \\

& & $(1,2,\cdots,i-1,i+1,\cdots,n+1,i)$ & \\

& & $(1,2,\cdots,i-2,i+1,i-1,i,i+2,\cdots,n+1)$ & \\

& & \qquad $\vdots$ & \\

& & $(i+1,1,2,\cdots,i,i+2,\cdots,n+1)$ & \\

\midrule


\multirow[c]{13}{*}{$B_n(C_n)$} & \multirow[c]{9}{*}{$\mathcal{C}_i(1\leq i\leq n-1)$}
& $(1,2,\cdots,n-i-1,n-i+1,n-i,\cdots,n)$
& \multirow[c]{9}{*}{$2n-1$} \\

& & \qquad $\vdots$ & \\

& & $(1,2,\cdots,n-i-1,n-i+1,\cdots,n,n-i)$ & \\

& & $(1,2,\cdots,n-i-2,n-i+1,n-i-1,\cdots,n)$ & \\

& & \qquad $\vdots$ & \\

& & $(n-i+1,1,2,\cdots,n-i-2,n-i-1,\cdots,n)$ & \\

& & $(-(n-i+1),1,2,\cdots,n-i,n-i+2,\cdots,n)$ & \\

& & \qquad $\vdots$ & \\

& & $(1,2,\cdots,n-i,n-i+2,\cdots,n,-(n-i+1))$ & \\

\addlinespace[0.5ex] 
\cdashline{2-4}
\addlinespace[0.5ex]

& \multirow[c]{4}{*}{$\mathcal{C}_n$}
& $(-1,2,\cdots,n)$
& \multirow[c]{4}{*}{$n+1$} \\

& & \qquad $\vdots$ & \\

& & $(2,\cdots,n,-1)$ & \\

& & $(-2,-1,3,\cdots,n)$ & \\

\midrule


\multirow[c]{15}{*}{$D_n$} & \multirow[c]{7}{*}{$\mathcal{C}_i(1\leq i\leq n-2)$}
& $(1,2,\cdots,n-i-1,n-i+1,n-i,\cdots,n)$
& \multirow[c]{7}{*}{$n$} \\

& & \qquad $\vdots$ & \\

& & $(n-i+1,1,2,\cdots,n-i,n-i+2,\cdots,n)$ & \\

& & $(-(n-i+1),-1,2,\cdots,n-i,n-i+2,\cdots,n)$ & \\

& & $(1,2,\cdots,n-i-1,n-i+1,n-i+2,n-i,\cdots,n)$ & \\

& & \qquad $\vdots$ & \\

& & $(1,2,\cdots,n-i-1,n-i+1,n-i+2,\cdots,n,n-i)$ & \\

\addlinespace[0.5ex] 
\cdashline{2-4}
\addlinespace[0.5ex]

& \multirow[c]{4}{*}{$\mathcal{C}_{n-1}$}
& $(2,1,3,\cdots,n)$
& \multirow[c]{4}{*}{$n$} \\

& & \qquad $\vdots$ & \\

& & $(2,3,\cdots,n,1)$ & \\

& & $(-3,-2,1,4,\cdots,n)$ & \\

\addlinespace[0.5ex] 
\cdashline{2-4}
\addlinespace[0.5ex]

& \multirow[c]{4}{*}{$\mathcal{C}_n$}
& $(-2,-1,3,\cdots,n)$
& \multirow[c]{4}{*}{$n$} \\

& & \qquad $\vdots$ & \\

& & $(-2,3,\cdots,n,-1)$ & \\

& & $(3,-2,-1,4,\cdots,n)$ & \\

\end{longtable}
\endgroup

From \cite[Table 1]{BKOP}, we have the following result.


\begin{lem}\label{long}
    The longest element $w_0$ of a Weyl group can be expressed as the following:
    \begin{itemize}
        \item for type $A_{n}$,
        \begin{align*}
            w_{A_{n}}&=s_{n}(s_{n-1}s_n)\cdots (s_1s_2\cdots s_n)\\
            &=(n+1,n,...,1);
        \end{align*}
        \item for type $B_n$ or $C_n$,
        \begin{align*}
w_{B_{n}}=w_{C_{n}}=&s_n(s_{n-1}s_ns_{n-1})\cdots (s_ks_{k+1}\cdots s_{n-1} s_ns_{n-1}\cdots s_{k+1}s_k)\\
&\cdots (s_1s_2\cdots s_{n-1} s_ns_{n-1}\cdots s_2s_1)\\
=&(-1,-2,\dots,-n);
        \end{align*}


        \item for type $D_n$,
        \begin{align*}
w_{D_{n}}=& s_ns_{n-1}(s_{n-2}s_ns_{n-1}s_{n-2})\cdots(s_k\cdots s_{n-2}s_ns_{n-1}s_{n-2}\cdots s_k)\\
&\cdots(s_1\cdots s_{n-2}s_n s_{n-1}\cdots s_{1})\\
=&\begin{cases}
    (-1,-2,\dots,-n) & \text{~if~$n$~is~even},\\
    (1,-2,-3,\dots,-n) &\text{~if~$n$~is~odd}.
\end{cases}
        \end{align*}

        \item for type $E_6$,
\begin{align*}
w_{E_{6}} =& s_1s_2s_3s_1s_4s_2s_3s_1s_4s_3s_5s_4s_2s_3s_1s_4s_3s_5\\
\cdot &s_4s_2s_6s_5s_4s_2s_3s_1s_4s_3s_5s_4s_2s_6s_5s_4s_3s_1;
\end{align*}

     \item for type $E_7$,
     \begin{align*}
w_{E_{7}} =& s_1s_2s_3s_1s_4s_2s_3s_1s_4s_3s_5s_4s_2s_3s_1s_4s_3s_5s_4s_2s_6\\
\cdot &s_5s_4s_2s_3s_1s_4s_3s_5s_4s_2s_6s_5s_4s_3s_1s_7s_6s_5s_4s_2s_3\\
\cdot &s_1s_4s_3s_5s_4s_2s_6s_5s_4s_3s_1s_7s_6s_5s_4s_2s_3s_4s_5s_6s_7;
        \end{align*}

        \item for type $E_8$,
        \begin{align*}
w_{E_{8}}=&s_1s_2s_3s_1s_4s_2s_3s_1s_4s_3s_5s_4s_2s_3s_1s_4s_3s_5s_4s_2\\
\cdot &s_6s_5s_4s_2s_3s_1s_4s_3s_5s_4s_2s_6s_5s_4s_3s_1s_7s_6s_5s_4\\
\cdot &s_2s_3s_1s_4s_3s_5s_4s_2s_6s_5s_4s_3s_1s_7s_6s_5s_4s_2s_3s_4\\
\cdot &s_5s_6s_7s_8s_7s_6s_5s_4s_2s_3s_1s_4s_3s_5s_4s_2s_6s_5s_4s_3\\
\cdot &s_1s_7s_6s_5s_4s_2s_3s_4s_5s_6s_7s_8s_7s_6s_5s_4s_2s_3s_1s_4\\
\cdot &s_3s_5s_4s_2s_6s_5s_4s_3s_1s_7s_6s_5s_4s_2s_3s_4s_5s_6s_7s_8;
      \end{align*}
\item for type $F_4$,
        \begin{align*}
 w_{F_4}=&s_3s_{2}s_{3}s_2 s_{1}s_{2}s_3 s_{2}s_1 s_{4}s_{3} s_2s_1 s_3s_2 s_{3} s_{4}s_3 \\
 \cdot &
s_2s_{1} s_3s_2s_3s_4;
        \end{align*}
\item for type $G_2$,
        \begin{align*}
 w_{G_2}=&s_1s_{2}s_{1}s_2 s_{1}s_{2}\\
 =&s_2s_1s_{2}s_{1}s_2 s_{1}.
        \end{align*}

    \end{itemize}
\end{lem}

The following result will be useful in the determination of smoothness of Schubert varieties of classical types.

\begin{Thm} \label{pattern-avoid}
    Let $\Phi$ be a root system of type $B$, $C$ or $D$. Then the Schubert variety $X(w)$ for $w\in w_0\mathcal{C}$ is smooth if and only if one of the following holds:
    \begin{enumerate}
        \item $\Phi$ is of type $B_n$, then $w\in w_0\mathcal{C}$ avoids the following patterns:
        \begin{align*}
            &\bar{1}2\bar{3}\quad 1\bar{2}\bar{3}\quad 1\bar{3}\bar{2}\quad  \bar{2}1\bar{3}\quad 2\bar{1}\bar{3}\quad  \\ & 2\bar{3}\bar{1}\quad
            \bar{3}1\bar{2}\quad \bar{3}\bar{2}1\quad
           \bar{3}2\bar{1}\quad \bar{3}\bar{4}\bar{1}\bar{2}\quad \bar{2}\bar{1}
        \end{align*}

 \item $\Phi$ is of type $C_n$, then $w\in w_0\mathcal{C}$ avoids the following patterns:
        \begin{align*}
            &\bar{1}2\bar{3}\quad \bar{2}\bar{1}\bar{3}\quad \bar{2}1\bar{3}\quad 2\bar{1}\bar{3}\quad 2\bar{3}\bar{1}\quad \\ & \bar{3}\bar{2}\bar{1}\quad \bar{3}\bar{2}1\quad
           \bar{3}2\bar{1}\quad 3\bar{2}\bar{1}\quad \bar{3}\bar{4}\bar{1}\bar{2}\quad 1\bar{2}
        \end{align*}

 \item $\Phi$ is of type $D_n$, then $w\in w_0\mathcal{C}$ avoids the following patterns:
\begin{enumerate}
    \item If $n$ is even:
 \begin{align*}
            &\bar{1}\bar{3}\bar{2}\quad 21\bar{3}\bar{4}\quad
            \bar{2}\bar{1}\bar{3}\quad
            2\bar{3}1\bar{4}\quad 3\bar{1}\bar{2}\bar{4}\quad
             \\&
             31\bar{2}\bar{4}\quad
            3\bar{2}1\bar{4}\quad
            \bar{3}\bar{2}\bar{1}\quad
            3\bar{2}\bar{4}1\quad \bar{3}\bar{4}\bar{1}\bar{2}\quad 3\bar{4}1\bar{2}
        \end{align*}

          \item If $n$ is odd:
\begin{align*}
            & \bar{1}\bar{3}\bar{2}\quad 1\bar{3}\bar{2}\quad 
            \bar{2}\bar{1}\bar{3}\quad \bar{2}\bar{3}1\bar{4}\quad \bar{2}1\bar{3}\bar{4}\quad 21\bar{3}\bar{4}\quad \\ &2\bar{1}\bar{3}\bar{4}\quad  3\bar{1}\bar{2}\bar{4}\quad 31\bar{2}\bar{4}\quad  3\bar{4}\bar{1}\bar{2}\quad \bar{3}\bar{4}1\bar{2}\quad \bar{3}\bar{4}\bar{2}1
        \end{align*}

\end{enumerate}
    \end{enumerate}
\end{Thm}
\begin{proof}
    There are $27$ patterns in the signed notation in type $B$ and $55$ patterns in type $D$ listed in \cite{Bi-98}. By \cite[Thm. 4.2, Thm. 6.1 and  Thm. 6.2]{Bi-98}, we may exclude those patterns that contain at least two positive elements  from Table \ref{Elements-in-cell} and Lemma \ref{long}. We can also exclude the patterns that contain the subpattern $\bar{2}\bar{1}$, except for the pattern $\bar{2}\bar{1}$ itself, since if a pattern occurring in $w$ contains $\bar{2}\bar{1}$ as a subpattern, then $w$ contains the pattern $\bar{2}\bar{1}$.
    For type $C$, the analogous exclusion applies with the subpattern $1\bar{2}$. Thus there are $11$ patterns in the signed notation in type $B$ or $C$.

    For type $D$, when $n$ is even, based on Table \ref{Elements-in-cell} and Lemma \ref{long}, we can exclude the patterns if they satisfy the following conditions:
    \begin{enumerate}
        \item Patterns in which the number of positive elements is larger than $2$.

        \item The number of positive elements in the pattern is $2$, and the first positive number is less than the second positive number.

        \item The number of positive elements in the pattern is $1$ or $2$, and the first  number on the left is negative.

       \item The first positive number is the largest of the absolute values of all the numbers in the pattern.

    \end{enumerate}
In the end, there are only $11$ patterns left.

For type $D$, when $n$ is odd, based on Table \ref{Elements-in-cell} and Lemma \ref{long}, we can exclude the patterns if they satisfy the following conditions:
  \begin{enumerate}
        \item Patterns in which the number of positive elements is $4$.

        \item Those patterns that start with $4$.

        \item Those patterns that end with $2$.

        \item The number of positive elements in the pattern is $2$ or $3$, and the first  number on the left is negative.

        \item The pattern contains exactly $2$ non-adjacent positive elements, separated by one or more negative numbers.

        \item The number of positive elements in the pattern is $2$, and the first positive number is less than the second positive number.

        \item The pattern contains exactly one positive number other than 1, and it does not appear in the first position on the left.

        \item The pattern is composed solely of negative numbers and cannot be transformed into an increasing sequence by removing one element.

    \end{enumerate}
In the end, there are only $12$ patterns left.

\end{proof}

\section{The case of $A_{n}$}\label{section-A}

In this section, we give the proof of our Theorem \ref{smcell} and Theorem \ref{AV-min} for type $A$.

{By \cite[Prop.~12]{Bchen}, two permutations lie in the same KL right cell if and only if their insertion tableaux agree.  Put $N=n+1$, denote
\[
x_3=(N,N-1,\ldots,3,1,2),\quad
x_4=(N-1,N-2,\ldots,1,N).
\]
}

{\bf Case 1.}
From \S \ref{KL-cells} and Lemma \ref{long}, we have \begin{align*}
    w_0\mathcal{C}_1=\{&(n,n+1,n-1,n-2,\dots,3,2,1),\\
    &(n,n-1,n+1,n-2,n-3,\dots,3,2,1),\cdots,\\
    &(n,n-1,n-2,\dots,3,2,1,n+1)\}.
\end{align*}
{The element $x_4=(n,n-1,\ldots,2,1,n+1)$ belongs to this right cell. By \cite[Prop.~12]{Bchen}, every $w\in w_0\mathcal{C}_1$ has $P(w)=P(x_4)$, where}
\[
{
P(x_4)=
\begin{tikzpicture}[scale=0.82,baseline=-44pt]
    \hobox{0}{0}{1}
    \hobox{1}{0}{n+1}
    \hobox{0}{1}{2}
    \hobox{0}{2}{\vdots}
    \hobox{0}{3}{n}
\end{tikzpicture}.
}
\]
{Since $x_4$ is listed in \cite[Cor.~6]{Bchen}, the right cell
$w_0\mathcal{C}_1$ is smooth.} Proposition \ref{smoothequal} and
Lemma \ref{constant} therefore give
\[
V(L_w)=\mathcal{V}(w)=\mathcal{V}(w_0s_1)
\]
for any $w\in w_0\mathcal{C}_1$.

{\bf Case 2.}
{We have $P(w_0s_n)=P(x_3)$ where
$x_3=(n+1,n,\ldots,3,1,2)$. Since $x_3$ occurs in the list of
\cite[Cor.~6]{Bchen} and $w_0s_n\in w_0\mathcal{C}_n$, the right cell
$w_0\mathcal{C}_n$ is smooth.}
Thus we have
\[
V(L_w)=\mathcal{V}(w_0s_n)
\]
for any $w\in w_0\mathcal{C}_n$.

{\bf Case 3.}
{For $1<i<n$, to determine the smooth elements in $w_0\mathcal{C}_i$ we use the pattern avoidance criterion directly in these cases.}
Note that \begin{align*}
    w_0s_is_{i+1}\cdots s_{k}=(&n+1,n,n-1,\cdots,n-i+3,n-i+1,n-i,\cdots,\\
&n-k+1, n-i+2,n-k,n-k-1,\cdots,2,1),
\end{align*}
which has the pattern $4231$ for $i\leq k\leq n-1$.

Note that \begin{align*}
    w_0s_is_{i-1}\cdots s_{k}=(&n+1,n,n-1,\cdots,n-k+3,n-i+1,n-k+2,\\
&n-k+1,\cdots,n-i+2,n-i,n-i-1,\cdots,2,1),
\end{align*}
which has the pattern $4231$ for $2\leq k\leq i$.

Also we have \begin{align*}
    w_0s_is_{i-1}\cdots s_{1}=(&n-i+1,n+1,n,n-1,\cdots,n-i+2,n-i,\\
    &n-i-1,\cdots,2,1),
\end{align*}
which avoids the two patterns $4231$ and $3412$. Thus by Proposition \ref{3412}, the Schubert variety $X(w_0s_is_{i-1}\cdots s_{1})$ is smooth.

{The other endpoint element is}
\[
w_0s_is_{i+1}\cdots s_{n}
=(n+1,n,\cdots,n-i+3,n-i+1,n-i,\cdots,2,1,n-i+2),
\]
which avoids the two patterns $4231$ and $3412$. Thus by Proposition \ref{3412}, the Schubert variety $X(w_0s_is_{i+1}\cdots s_{n})$ is smooth.

Therefore, from Lemma \ref{constant}, we have
\[
V(L_w)=V(L_{w_0s_is_{i+1}\cdots s_{n}})
=\mathcal{V}(w_0s_is_{i+1}\cdots s_{n})
\]
for any $w\in w_0\mathcal{C}_i$.

For a right cell $\mathcal{R}$, we use $S(\mathcal{R})$ to denote the smooth elements contained in $\mathcal{R}$.
\begin{cor}
For type $A_{n}$, we have $S(w_0\mathcal{C}_{1})=w_0\mathcal{C}_{1}$ and $S(w_0\mathcal{C}_{n})=w_0\mathcal{C}_{n}$. When $2\leq i\leq n-1$, we have
\[
S(w_0\mathcal{C}_{i})=\{w_0s_is_{i-1}\cdots s_{1},w_0s_is_{i+1}\cdots s_{n}\}.
\]
\end{cor}

\begin{cor}
For type $A_{n}$ and $w\in w_0\mathcal{C}$, the associated variety of $L_w$ is irreducible. We have
\[
V(L_w)=\mathcal{V}(w_0s_is_{i+1}\cdots s_{n})
=\mathcal{V}(w_0s_is_{i-1}\cdots s_{1})
\]
if $w\in w_0\mathcal{C}_{i}$.
\end{cor}

By Proposition \ref{geometriccell}, this common orbital variety is
$\mathcal{V}(w_0s_i)=\mathcal{V}(w_{\rm max})$.

\section{The case of $C_n$ and $B_n$}\label{typebc}

In this section, we give the proof of our Theorem \ref{smcell} and Theorem \ref{AV-min} for types $B$ and $C$.

\subsection{The case of $C_n$}
{We determine the smooth elements in each right cell by considering
the possible reduced expressions.}

{\bf Case 1.} From Table \ref{Elements-in-cell} and Lemma \ref{long}, we have
\begin{align*}
    w_0\mathcal{C}_1=\{&(-1,-2,\cdots,2-n,-n,1-n),\\
    &(-1,-2,\cdots,3-n,-n,2-n,1-n),\\
	&\vdots\\
    &(-n,-1,-2,\cdots,1-n),\\
	&(n,-1,-2,\cdots,1-n),\\
	&(-1,n,-2,\cdots,1-n),\\
	&\vdots\\
	&(-1,-2,\cdots,1-n,n)\}.
\end{align*}



Next we prove that every element in this KL right cell $w_0\mathcal{C}_1$ is smooth. We use $[n, m]$ to denote the set $\{n, n + 1, ..., m\}$ for $n < m$. We use the notation of Definition \ref{signed-notation}
and write $\bar{a}=-a$.

For $w_0\mathcal{C}_1$, we have
\begin{align*}
    w_0s_1s_2\cdots s_k=(&-1,-2,\cdots,-n+k+1,-n,-n+k,\\&-n+k-1,\cdots,1-n)\text{~for~} 1\leq k\leq n-1
\end{align*}
 and
 \begin{align*}
     w_0s_1s_2\cdots s_ns_{n-1}\cdots s_{n-j+1}=(&-1,-2,\cdots,1-j,n,-j,\\&-j-1,\cdots,1-n)\text{~for~}1\leq j\leq n.
 \end{align*}

 For the permutation $(-1,-2,\cdots,-n+k+1,-n,-n+k,-n+k-1,\cdots,1-n)$, we assert that the patterns $\bar{2}\bar{1}\bar{3}$, $\bar{3}\bar{2}\bar{1}$ and $\bar{3}\bar{4}\bar{1}\bar{2}$ can not appear. Note that the subsequence $(-1,-2,\cdots,-n+k+1,-n+k,-n+k-1,\cdots,1-n)$ is decreasing, then the pattern $\bar{2}\bar{1}\bar{3}$ and $\bar{3}\bar{2}\bar{1}$ can not appear. For the pattern $\bar{3}\bar{4}\bar{1}\bar{2}$, the element “$\bar{3}$”
must be chosen from $[-n,-1]$ (since “$\bar{1}$” and “$\bar{2}$” are fixed at $[1-n,k-n]$). The relative order of “$\bar{3}$” and “$\bar{1}$” makes the pattern $\bar{3}\bar{4}\bar{1}\bar{2}$  impossible.

For the permutation $(-1,-2,\cdots,1-j,n,-j,-j-1,\cdots,1-n)$, the discussion is similar and next we assert that the patterns $\bar{1}2\bar{3}$, $\bar{2}1\bar{3}$, $2\bar{1}\bar{3}$, $2\bar{3}\bar{1}$, $\bar{3}\bar{2}1$, $\bar{3}2\bar{1}$, $3\bar{2}\bar{1}$, $1\bar{2}$ can not appear. Note that the positive part of these eight patterns is fixed by $n$ and  the subsequence $(-1,-2,\cdots,1-j,-j,-j-1,\cdots,1-n)$ is decreasing, then the patterns $2\bar{3}\bar{1}$, $\bar{3}\bar{2}1$, $\bar{3}2\bar{1}$, $3\bar{2}\bar{1}$ can not appear. The remaining four patterns $\bar{1}2\bar{3}$, $\bar{2}1\bar{3}$, $2\bar{1}\bar{3}$ and $1\bar{2}$ can not appear since there is no negative element with a larger absolute value than $n$ in this permutation.

Therefore by Theorem \ref{pattern-avoid},  the Schubert variety $X(w)$  is smooth for each $w\in w_0\mathcal{C}_1$.

{\bf Case 2.} Similarly in the KL right cell $w_0\mathcal{C}_n$, we have
$$w_0s_ns_{n-1}\cdots s_1=(-2,\cdots,-n,1),$$
which avoids the  $11$ patterns listed in Theorem \ref{pattern-avoid}, thus the Schubert variety $X(w_0s_ns_{n-1}\cdots s_1)$ is smooth.

Also, we have
$$
w_0s_ns_{n-1}\cdots s_k=(-2,\cdots,-(n-k+1),1,-(n-k+2),\cdots,-n),
$$
which has the pattern $1\bar{2}$ for $2\leq k\leq n-1$. The same pattern also appears in $ w_0s_ns_{n-1}s_n=(2,1,-3,-4,\cdots,-n)$ and in $ w_0s_n=(1,-2,\cdots,-n)$.

{\bf Case 3.} For any $1<i<n$ and $w\in w_0\mathcal{C}_i$, we have the following.
\begin{enumerate}
    \item  For $w=w_0s_is_{i-1}\cdots s_{i-k+1}$,
    \begin{enumerate}
        \item When $1\leq k\leq i-1$, we have
\begin{align*}
&w_0s_is_{i-1}\cdots s_{i-k+1}\\
=&(-1,-2,\cdots,-(n-i-1),-(n-i+1),\cdots,\\&-(n-i+k),-(n-i),-(n-i+k+1),\cdots,-n),
\end{align*}
which has the pattern $\bar{2}\bar{1}\bar{3}$.

\item  When $k=i$, we have
\begin{align*}
&w_0s_is_{i-1}\cdots s_{1}\\
=&(-1,-2,\cdots,-(n-i-1),-(n-i+1),\cdots,-n,-(n-i)),
\end{align*}
which avoids the $11$ patterns listed in Theorem \ref{pattern-avoid}.
Thus $X(w_0s_is_{i-1}\cdots s_1)$ is smooth.
    \end{enumerate}

\item  For $w=w_0s_is_{i+1}\cdots s_k$ with $i\leq k\leq n-1$, we have
\begin{align*}
&w_0s_is_{i+1}\cdots s_k\\
=&(-1,-2,\cdots,-(n-k-1),-(n-i+1),-(n-k),\\
&-(n-k+1),\cdots,-(n-i),-(n-i+2),\cdots,-n),
\end{align*}
which has the pattern $\bar{2}\bar{1}\bar{3}$.

\item For $w=w_0s_is_{i+1}\cdots s_ns_{n-1}\cdots s_{n-k+1}$,
\begin{enumerate}
    \item When $1\leq k\leq n-i+1$, we have
\begin{align*}
&w_0s_i s_{i+1}\cdots s_n s_{n-1}\cdots s_{n-k+1}\\
&=(-1,-2,\ldots,-(k-1),n-i+1,-k,\ldots,-(n-i),-(n-i+2),\ldots,-n).
\end{align*}
which has the pattern $1\bar{2}$.

\item When $n-i+1<k\leq n-1$, we have
\begin{align*}
&w_0s_i s_{i+1}\cdots s_n s_{n-1}\cdots s_{n-k+1}\\
&=(-1,-2,\ldots,-(n-i),-(n-i+2),\ldots,-k, n-i+1,-(k+1),\ldots,-n).
\end{align*}
which has the pattern $\bar{2}1\bar{3}$.

\item When $k=n$,  we have
\begin{align*}
&w_0s_is_{i+1}\cdots s_ns_{n-1}\cdots s_{1}\\
=&(-1,-2,\ldots,-(n-i),-(n-i+2),\ldots,-n,\\
&\qquad n-i+1),
\end{align*}
which avoids the $11$ patterns listed in Theorem \ref{pattern-avoid}.
Thus $X(w_0s_is_{i+1}\cdots s_ns_{n-1}\cdots s_1)$ is smooth.

\end{enumerate}

\end{enumerate}





\begin{cor}
    For type $C_n$ with $n\geq3$, we have $S(w_0\mathcal{C}_1)=w_0\mathcal{C}_1$ and  $S(w_0\mathcal{C}_n)=\{w_0s_ns_{n-1}\cdots s_1\}$.  When $1<i< n$, we have
    $$S(w_0\mathcal{C}_i)=\{w_0s_is_{i-1}\cdots s_1, w_0s_is_{i+1}\cdots s_ns_{n-1}\cdots s_{1}\}.$$
\end{cor}

\begin{cor}
    When $W=W(C_n)$ and $w\in w_0\mathcal{C}$, the associated variety of $L_w$ will be irreducible. We have $V(L_w)=\mathcal{V}(w_0s_is_{i+1}\cdots s_ns_{n-1}\cdots s_{1})$ if $w\in w_0\mathcal{C}_i$.
\end{cor}

Note that $w_0s_is_{i+1}\cdots s_ns_{n-1}\cdots s_{1}$ is a minimal-length element $w_{\rm min} $
 of the KL right cell $w_0\mathcal{C}_i$.

\subsection{The case of $B_n$}
Now we consider the case of type $B_n$. The result for type $B_n$ can be obtained by adapting the argument for type $C_n$, since they share the same Weyl group. From \cite[Thm. 4.2 and Thm. 6.1]{Bi-98} we divide the discussion into the following three cases:
\begin{enumerate}
    \item[(i)] If $X(w)$ is not smooth in type $C_n$ and $w$ contains one of the patterns listed in \cite[Thm. 4.2]{Bi-98}, then $X(w)$ is not smooth in type $B_n$;

    \item[(ii)] If $X(w)$ is not smooth in type $C_n$ and $w$ contains only the pattern $1\bar{2}$, then to determine whether $X(w)$ is smooth in type $B_n$, it suffices to check whether $w$ avoids the pattern $\bar{2}\bar{1}$;

    \item[(iii)] If $X(w)$ is smooth in type $C_n$, which means that $w$ avoids the patterns listed in \cite[Thm. 4.2 and Thm. 6.1]{Bi-98}, then it again suffices to check whether $w$ avoids the pattern $\bar{2}\bar{1}$.
\end{enumerate}


{\bf Case 1.} For the KL right cell $ w_0\mathcal{C}_1$, we have $S( w_0\mathcal{C}_1)= w_0\mathcal{C}_1$ in the case of $C_n$, then we only need to check the pattern $\bar{2}\bar{1}$ in the case of type $B_n$.

For $w\in w_0\mathcal{C}_1$, we have
\begin{align*}
    w_0s_1s_2\cdots s_k=(&-1,-2,\cdots,-n+k+1,-n,-n+k,\\&-n+k-1,\cdots,1-n)\text{~for~} 1\leq k\leq n-1
\end{align*}
 and
 \begin{align*}
     w_0s_1s_2\cdots s_ns_{n-1}\cdots s_{n-j+1}=(&-1,-2,\cdots,1-j,n,-j,\\&-j-1,\cdots,1-n)\text{~for~}1\leq j\leq n.
 \end{align*}

For the permutation of $(-1,-2,\cdots,-n+k+1,-n,-n+k,-n+k-1,\cdots,1-n)$, it has the pattern $\bar{2}\bar{1}$ for $1\leq k \leq n-1$. For the permutation $(-1,-2,\cdots,1-j,n,-j,-j-1,\cdots,1-n)$, it avoids the pattern $\bar{2}\bar{1}$ for $1\leq j\leq n$ since the subsequence $(-1,-2,\cdots,1-j,-j,-j-1,\cdots,1-n)$ is decreasing, the relative order of “$\bar{2}$” and “$\bar{1}$” makes the pattern $\bar{2}\bar{1}$ impossible.

{\bf Case 2.} Similarly for the KL right cell $w_0\mathcal{C}_n$, we have $S(w_0\mathcal{C}_n)=\{w_0s_ns_{n-1}\cdots s_1\}$ in the case of $C_n$. Note that
$$w_0s_ns_{n-1}\cdots s_1=(-2,\cdots,-n,1),$$
which avoids the pattern $\bar{2}\bar{1}$, thus by Theorem \ref{pattern-avoid}  the Schubert variety $X(w_0s_ns_{n-1}\cdots s_1)$ is smooth.  \\
We also have
\begin{align*}
w_0s_ns_{n-1}\cdots s_k=(-2,\cdots,-(n-k+1),1,-(n-k+2),\cdots,-n),
\end{align*}
which has the pattern $\bar{2}1\bar{3}$ for $2\leq k\leq n-1$, and
$$w_0s_ns_{n-1}s_n=(2,1,-3,-4,\cdots,-n)$$
has the pattern $1\bar{2}\bar{3}$ when $n\geq 4$. Furthermore, 
$$w_0s_n=(1,-2,\cdots,-n)$$
has the pattern $1\bar{2}\bar{3}$.

{\bf Case 3.} For any $1<i<n$ and $w\in w_0\mathcal{C}_i$, we have the following.
\begin{enumerate}
    \item For $w=w_0s_is_{i-1}\cdots s_{i-k+1}$, when $1\leq k\leq i$ we have
\begin{align*}
&w_0s_is_{i-1}\cdots s_{i-k+1}\\
=&(-1,-2,\cdots,-(n-i-1),-(n-i+1),\cdots,-(n-i+k),\\
&-(n-i),-(n-i+k+1),\cdots,-n),
\end{align*}
which has the pattern $\bar{2}\bar{1}$.

\item For $w=w_0s_is_{i+1}\cdots s_k$, when $i\leq k\leq n-1$ we have
\begin{align*}
&w_0s_is_{i+1}\cdots s_k\\
=&(-1,-2,\cdots,-(n-k-1),-(n-i+1),-(n-k),\\&-(n-k+1),\cdots,-(n-i),-(n-i+2),\cdots,-n),
\end{align*}
which has the pattern $\bar{2}\bar{1}$.

\item For $w=w_0s_is_{i+1}\cdots s_ns_{n-1}\cdots s_{n-k+1}$,
\begin{enumerate}
    \item When $k=1$, we have
    \begin{align*}
w_0s_is_{i+1}\cdots s_n=(n-i+1,-1,-2,\cdots,-(n-i),-(n-i+2),\cdots,-n),
\end{align*}
which has the pattern  $2\bar{1}\bar{3}$.

 \item When $2\leq k\leq n-i+1$, we have
\begin{align*}
&w_0s_i s_{i+1}\cdots s_n s_{n-1}\cdots s_{n-k+1}\\
&=(-1,-2,\ldots,-(k-1),n-i+1,-k,\ldots,-(n-i),-(n-i+2),\ldots,-n).
\end{align*}
which has the pattern $\bar{1}2\bar{3}$.

\item When $n-i+1<k\leq n-1$, we have
\begin{align*}
&w_0s_i s_{i+1}\cdots s_n s_{n-1}\cdots s_{n-k+1}\\
&=(-1,-2,\ldots,-(n-i),-(n-i+2),\ldots,-k, n-i+1,-(k+1),\ldots,-n).
\end{align*}
which has the pattern $\bar{2}1\bar{3}$.

\item When $k=n$, we have
\begin{align*}
&w_0s_is_{i+1}\cdots s_ns_{n-1}\cdots s_1\\
=&(-1,-2,\ldots,-(n-i),-(n-i+2),\ldots,-n,\\
&\qquad n-i+1),
\end{align*}
which avoids the pattern $\bar{2}\bar{1}$. By Theorem
\ref{pattern-avoid}, the Schubert variety associated with this element
is therefore smooth.

\end{enumerate}

\end{enumerate}





\begin{cor}
    For type $B_n$ with $n\geq 4$, we have $$S(w_0\mathcal{C}_1)=\{w_0s_1\cdots s_ns_{n-1}\cdots s_{n-j+1}\mid 1\leq j\leq n\}$$ and  $$S(w_0\mathcal{C}_n)=\{w_0s_ns_{n-1}\cdots s_1\}.$$   When $1<i< n$, we have
    $$S(w_0\mathcal{C}_i)=\{w_0s_is_{i+1}\cdots s_ns_{n-1}\cdots s_1\}.$$
\end{cor}

\begin{cor}
    When $W=W(B_n)$ and $w\in w_0\mathcal{C}$, the associated variety of $L_w$ will be irreducible. We have
$V(L_w)=\mathcal{V}(w_0s_is_{i+1}\cdots s_ns_{n-1}\cdots s_1)$ if $w\in w_0\mathcal{C}_i$.
\end{cor}

Note that $w_0s_is_{i+1}\cdots s_ns_{n-1}\cdots s_{1}$ is a minimal-length element $w_{\rm min} $
 of the KL right cell $w_0\mathcal{C}_i$.

\begin{Rem}\label{low-rank-BC}
Theorem~\ref{AV-min}(1) also holds in types $B_2$, $C_2$, and $B_3$.
 In type $B_3$,
\[
S(w_0\mathcal C_3)=
\{w_0s_3s_2s_1,\ w_0s_3s_2s_3\},
\]
as follows by applying Theorem~\ref{BP-patterns} to the elements
listed in Lemma~\ref{rightcell-ci}.

In types $B_2$ and $C_2$, we may choose
\[
w_{\min,1}=w_0s_1s_2s_1\in w_0\mathcal{C}_1,\qquad w_{\min,2}=w_0s_2s_1s_2\in w_0\mathcal{C}_2.
\]
In type $B_3$, we may choose
\[
w_{\min,1}=w_0s_1s_2s_3s_2s_1\in w_0\mathcal{C}_1,\qquad w_{\min,2}=w_0s_2s_3s_2s_1\in w_0\mathcal{C}_2,\qquad w_{\min,3}=w_0s_3s_2s_1\in w_0\mathcal{C}_3.
\]
These elements have minimal length in their respective right cells
and are smooth. Hence Proposition~\ref{smoothequal} and
Lemma~\ref{constant} give
\[
V(L_w)=\mathcal{V}(w_{\min,i})
\qquad\text{for all }w\in w_0\mathcal{C}_i.
\]
\end{Rem}

\section{The case of $D_n$}\label{typed}

In this section, we give the proof of our Theorem \ref{smcell} and Theorem \ref{AV-min} for type $D$.

Before proving our results, we have the following special case.
\begin{prop}\label{non-smooth}
    For type $D_n$ $(n\geq 4)$, all elements of the following KL right cells are non-smooth: $\{w_0\mathcal{C}_i\mid 2\leq i \leq n-2\}$.
\end{prop}
\begin{proof}
{\bf Case 1.} Suppose that $n$ is even.

    Note that
    \begin{align*}
        &w_0s_is_{i-1}\cdots s_k\\
        =&(-1,\cdots,-(n-i-1),-(n-i+1),\cdots,-(n-k+1),\\&-(n-i),-(n-k+2),\cdots,-n),
    \end{align*}
which has the pattern $\bar{1}\bar{3}\bar{2}$ for $1\leq k\leq i-1$.

For the permutation
\begin{align*}
    &w_0s_is_{i+1}\cdots s_{n-2}s_n\\=&(n-i+1,1,-2,\cdots,-(n-i-1),-(n-i),\\&-(n-i+2),\cdots,-n),
\end{align*}
which has the pattern $31\bar{2}\bar{4}$.

 Note that
    \begin{align*}
        &w_0s_is_{i+1}\cdots s_{i+(k-1)}\\
        =&(-1,\cdots,-(n-i-k),-(n-i+1),-(n-i-k+1),\\&-(n-i-k+2),\cdots,-(n-i),-(n-i+2),\cdots,-n),
    \end{align*}
which has the pattern $\bar{1}\bar{3}\bar{2}$ for $2\leq k< n-i$, and the pattern $\bar{2}\bar{1}\bar{3}$ for $k=1$ and $n-i$.

{\bf Case 2.} Suppose that $n$ is odd.

  Note that
    \begin{align*}
        &w_0s_is_{i-1}\cdots s_k\\
        =&(1,\cdots,-(n-i-1),-(n-i+1),\cdots,-(n-k+1),\\&-(n-i),-(n-k+2),\cdots,-n),
    \end{align*}
which has the pattern $1\bar{3}\bar{2}$ for $1\leq k\leq i-1$.

For the permutation
\begin{align*}
    w_0s_is_{i+1}\cdots s_{n-2}s_n=(&n-i+1,-1,-2,\cdots,-(n-i-1),-(n-i),\\&-(n-i+2),\cdots,-n),
\end{align*}
 which has the pattern $3\bar{1}\bar{2}\bar{4}$.

 Note that
    \begin{align*}
        &w_0s_is_{i+1}\cdots s_{i+(k-1)}\\
        =&(1,\cdots,-(n-i-k),-(n-i+1),-(n-i-k+1),\\&-(n-i-k+2),\cdots,-(n-i),-(n-i+2),\cdots,-n),
    \end{align*}
which has the pattern $1\bar{3}\bar{2}$ for $1\leq k\leq n-i-1$. Also
$$w_0 s_i s_{i+1} \cdots s_{n-1} = (-(n - i + 1), 1, -2, \dots, -(n - i), -(n - i + 2), \dots, -n),$$
which has the pattern $\bar{2}\bar{1}\bar{3}$.

\end{proof}

For the remaining cases, we split the discussion into two subcases based on the parity of \(n\).
\subsection{Proof for even $n$}

{\bf Case 1.} From Table \ref{Elements-in-cell} and Lemma \ref{long}, we have
\begin{align*}
    w_0\mathcal{C}_1=\{&(-1,-2,\cdots,2-n,-n,1-n),\\
	&\vdots\\
    &(-n,-1,-2,\cdots,1-n),\\
	&(n,1,-2,\cdots,1-n)\}.
\end{align*}

For the first $n-2$ permutations, they have the pattern $\bar{1}\bar{3}\bar{2}$. Next we prove that the last two permutations avoid the  $11$ patterns listed in Theorem \ref{pattern-avoid}.

For the permutation $(-n,-1,-2,\cdots,1-n)$, we only need to check the patterns $\bar{3}\bar{4}\bar{1}\bar{2}$, $\bar{1}\bar{3}\bar{2}$, $\bar{2}\bar{1}\bar{3}$ and $\bar{3}\bar{2}\bar{1}$. For the pattern $\bar{3}\bar{4}\bar{1}\bar{2}$, the element “$\bar{1}$” and “$\bar{2}$”
must be chosen from $[1-n,-1]$. The relative order of “$\bar{3}$” and “$\bar{4}$” makes the pattern $\bar{3}\bar{4}\bar{1}\bar{2}$  impossible. For the pattern $\bar{1}\bar{3}\bar{2}$, the element “$\bar{1}$” and “$\bar{3}$” must be chosen from $[-n,-1]$, but there is no larger element “$\bar{2}$” such that “$\bar{2}$” is on the right side of  “$\bar{3}$”, thus the relative order of “$\bar{2}$” and “$\bar{3}$” makes the pattern $\bar{1}\bar{3}\bar{2}$ impossible. Similarly for the patterns $\bar{2}\bar{1}\bar{3}$ and $\bar{3}\bar{2}\bar{1}$, the relative order of “$\bar{2}$” and “$\bar{1}$" makes them impossible. The discussion for $(n,1,-2,\cdots,1-n)$ is similar.

{\bf Case 2.} For $ w_0\mathcal{C}_{i}$ with $2\leq i\leq n-2$, we have given the result in Proposition \ref{non-smooth}.

{\bf Case 3.} For $ w_0\mathcal{C}_{n-1}$, we have
   $$ w_0s_{n-1}s_{n-2}s_n=(3,2,-1,-4,\cdots,-n),$$
which has the pattern $21\bar{3}\bar{4}$ for $n>4$.

Also we have $$w_0s_{n-1}s_{n-2}\cdots s_k=(-2,-3,\cdots,-(n-k+1),-1,-(n-k+2),\cdots,-n),$$
which has the pattern $\bar{2}\bar{1}\bar{3}$ for $2\leq k\leq n-1$. Note that $$w_0s_{n-1}s_{n-2}\cdots s_1=(-2,\cdots,-n,-1),$$ which avoids the $11$ patterns listed in Theorem \ref{pattern-avoid}, thus the Schubert variety $X(w_0s_{n-1}s_{n-2}\cdots s_1)$ is smooth.

{\bf Case 4.} For $w_0\mathcal{C}_{n}$, we have
   $$ w_0s_ns_{n-2}s_{n-1}=(-3,2,1,-4,\cdots,-n),$$
which has the pattern $21\bar{3}\bar{4}$ for $n>4$.

Also we have $$w_0s_{n}s_{n-2}\cdots s_k=(2,-3,\cdots,-(n-k+1),1,-(n-k+2),\cdots,-n),$$
which has the pattern $2\bar{3}1\bar{4}$ for $2\leq k\leq n-2$ and $21\bar{3}\bar{4}$ for $k=n-1$. Note that $$w_0s_{n}s_{n-2}\cdots s_1=(2,-3,\cdots,-n,1),$$ which avoids the $11$ patterns listed in Theorem \ref{pattern-avoid}, thus the Schubert variety $X(w_0s_{n}s_{n-2}\cdots s_1)$ is smooth.

\subsection{Proof for odd $n$}

{\bf Case 1.} From Table \ref{Elements-in-cell} and Lemma \ref{long}, we have
\begin{align*}
    w_0\mathcal{C}_1=\{&(1,-2,\cdots,2-n,-n,1-n),\\
	&\vdots\\
    &(n,-1,-2,\cdots,1-n),\\
	&(-n,1,-2,\cdots,1-n)\}.
\end{align*}

For the first $n-2$ permutations, they contain the pattern $1\bar{3}\bar{2}$, and the last two permutations avoid the $12$ patterns listed in Theorem \ref{pattern-avoid}. The argument is similar to the case where $n$ is even, so we omit the details here.

{\bf Case 2.} For $ w_0\mathcal{C}_{i}$ with $2\leq i\leq n-2$, we have given the result in
Proposition \ref{non-smooth}.

{\bf Case 3.} For $ w_0\mathcal{C}_{n-1}$, we have
   $$ w_0s_{n-1}s_{n-2}s_n=(3,2,1,-4,\cdots,-n),$$
which has the pattern $21\bar{3}\bar{4}$.

Also we have $$w_0s_{n-1}s_{n-2}\cdots s_k=(-2,-3,\cdots,-(n-k+1),1,-(n-k+2),\cdots,-n),$$
which has the pattern $\bar{2}\bar{3}1\bar{4}$ for $2\leq k\leq n-2$ and $\bar{2}1\bar{3}\bar{4}$ for $k=n-1$. When $k=1$ we have $$w_0s_{n-1}s_{n-2}\cdots s_1=(-2,-3,\cdots,-n,1),$$ which avoids the $12$ patterns listed in Theorem \ref{pattern-avoid}, thus the Schubert variety $X(w_0s_{n-1}s_{n-2}\cdots s_1)$ is smooth.

{\bf Case 4.} For $ w_0\mathcal{C}_{n}$, we have
   $$ w_0s_ns_{n-2}s_{n-1}=(-3,2,-1,-4,\cdots,-n),$$
which has the pattern $\bar{2}\bar{1}\bar{3}$.

Also we have $$w_0s_{n}s_{n-2}\cdots s_k=(2,-3,\cdots,-(n-k+1),-1,-(n-k+2),\cdots,-n),$$
which has the pattern $\bar{2}\bar{1}\bar{3}$ for $2\leq k\leq n-2$ and $2\bar{1}\bar{3}\bar{4}$ for $k=n-1$. When $k=1$ we have $$w_0s_{n}s_{n-2}\cdots s_1=(2,-3,\cdots,-n,-1),$$ which avoids the $12$ patterns listed in Theorem \ref{pattern-avoid}, thus the Schubert variety $X(w_0s_{n}s_{n-2}\cdots s_1)$ is smooth.
\begin{Rem}
   In types $D_3$ and $D_4$, in addition to the aforementioned smooth elements, there are additional smooth ones in $w_0\mathcal{C}_{n-1}$ and $w_0\mathcal{C}_{n}$:
    \smallskip
\begin{center}
\begin{tabular}{c|c|c}
\hline
 Lie type  &  $w_0\mathcal{C}_{n-1}$ &  $w_0\mathcal{C}_{n}$ \\
\hline
\mystrut
$D_3$  & $w_0s_{2}s_{1}s_3$, $w_0s_{2}$ & $w_0s_{3}s_{1}s_2$, $w_0s_{3}$ \\[.05in]
\hline
\mystrut
$D_4$ & $w_0s_{3}s_{2}s_4$ & $w_0s_4s_2s_3$
   \\[.05in]
\hline

\end{tabular}
\end{center}
\medskip
\end{Rem}

\begin{cor}\label{smooth}
    For type $D_n$ ($n\geq 4$), we have
$$S(w_0\mathcal{C}_1)=\{w_0s_1\cdots s_{n-2}s_{n-1},\; w_0s_1\cdots s_{n-2}s_{n}\}.$$
Moreover, for $n\geq 5$, we have $$S(w_0\mathcal{C}_{n-1})=\{w_0s_{n-1}s_{n-2}\cdots s_1\}$$ and $$S(w_0\mathcal{C}_n)=\{w_0s_{n}s_{n-2}\cdots s_1\}.$$
\end{cor}

\begin{prop}
    When $W=W(D_n)$ with $n\geq 5$ and $w\in w_0\mathcal{C}_i$ for $1\leq i\leq n$, we have
    \[{
V(L_w)=
\begin{cases}
\mathcal V(s_1w_0)=\mathcal V(w_0s_1)=\mathcal V(w_0s_1\cdots s_{n-2}s_{n-1}) & \text{if~}i=1;\\
\mathcal V(s_iw_0)=\mathcal V(w_0s_i)=\mathcal V(w_0s_is_{i+1}\cdots s_{n-1})  & \text{if~} 2\leq i\leq n-2\\
&\text{~and~}n-i\geq i;\\
\mathcal V(s_iw_0)=\mathcal V(w_0s_i)=\mathcal V(w_0s_is_{i-1}\cdots s_{1})  & \text{if~} 2\leq i\leq n-2\\
&\text{~and~}n-i< i;\\
\mathcal V(w_0s_{n-1})= \mathcal V(w_0s_{n-1}s_{n-2}\cdots s_1)& \text{if~}i=n-1;\\
\mathcal V(w_0s_{n})=\mathcal V(w_0s_{n}s_{n-2}\cdots s_1) & \text{if~}i=n.
\end{cases}}
\]

\end{prop}
\begin{proof}
  Suppose $w\in w_0\mathcal{C}_i$. If the right cell $w_0\mathcal{C}_i$ contains some smooth element $w_0x$, we will have $V(L_w)=V(L_{w_0x})=\mathcal{V}(w_0x)$ for any $w\in w_0\mathcal{C}_i$ by Lemma \ref{constant} and Proposition \ref{smoothequal}. Thus
$V(L_{w})$ can be characterized by Corollary \ref{smooth} for $i=1, n-1$ and $n$.

From \cite[\S 2.3 and Exer. 4.10]{BB05}, we have $$w_0s_i=\begin{cases}
s_iw_0 &  \text{~if~} 1\leq i\leq n-2;\\
s_{n-1}w_0  &  \text{~if~$n$~is~even, and~} i=n-1;\\
s_{n}w_0  &  \text{~if~$n$~is~odd, and~} i=n-1;\\
s_{n}w_0  &  \text{~if~$n$~is~even, and~} i=n;\\
s_{n-1}w_0  &  \text{~if~$n$~is~odd, and~} i=n.
\end{cases}$$

 When $i=1$, similarly we also have $V(L_w)=\mathcal V(s_1w_0)=\mathcal V(w_0s_1)$.

 But when $i=n-1$, from Lemma \ref{minimal-integralequal} we have
 $V(L_w)=\mathcal V(s_{n-1}w_0)=\mathcal V(w_0s_{n-1})$ when $n$ is even and $V(L_w)=\mathcal V(s_{n}w_0)=\mathcal V(w_0s_{n-1})$ when $n$ is odd. When $i=n$, the result is similar.

 Next, we consider the case where there is no smooth element in $w_0\mathcal{C}_i$ for $2\leq i\leq n-2$.

 When $w \in w_0\mathcal{C}_i$ with $2\leq i\leq n-2$, Lemma \ref{constant} implies $V(L_w) = V(L_{w_i})$, where  $w_i=w_0s_is_{i+1}\cdots s_{n-1}$ if $n-i\geq i$ and $w_i=w_0s_is_{i-1}\cdots s_{1}$ if $n-i< i$.
 \begin{enumerate}
     \item When $n$ is even  and $n-i\geq i$, we have $w_0(\alpha_i)=-\alpha_i$. In this case $w_i=w_0s_is_{i+1}\cdots s_{n-1}$ is a minimal-length element in the right cell $w_0\mathcal{C}_i$. From Lemma \ref{gamma-w} we have $V(L_{w_i})=\bigcup\limits_{y\in \Gamma(w_i)}\mathcal{V}(y)$ for some subset $\Gamma(w_i)\subseteq  W$.  Also  for any $y\in \Gamma(w_i)$, we have
     \begin{equation}\label{orbital-element}
    y\leq w_i \text{~and~} y\in w_0\mathcal{C}.
     \end{equation}
By Proposition \ref{bruhat-order}, there are $i$ elements which may satisfy the condition (\ref{orbital-element}): $$w_i,w_{i-1},\cdots, w_1,$$ where $w_t=w_0s_ts_{t+1}\cdots s_{n-1}\in w_0\mathcal{C}_t$ for $1\leq t\leq i$. We claim that
     \begin{align}\label{Gamma-set}
         \Gamma(w_t)=\{w_t\}.
     \end{align}
We prove the claim (\ref{Gamma-set}) by induction on $t$. If $t=1$, $w_1$ is a minimal-length element in the two-sided cell $w_0\mathcal{C}$, then $ \Gamma(w_1)=\{w_1\}$ by Lemma \ref{min-length-w}. If $t=2$, by Proposition \ref{bruhat-order}, there are only two elements which may satisfy the condition (\ref{orbital-element}): $w_2$ and $w_1$. Denote ${\rm J}_l=\Pi\setminus \{\alpha_l\}$. First we have $-w_0s_2\rho=-w_0(\rho-\alpha_2)=\rho-\alpha_2$, which is $\Phi_{{\rm J}_{2}}^+$-dominant. Then from Proposition \ref{Same}, $V(L_{w_2})=V(L_{w_0s_2})\subseteq  \fru_{\operatorname{J}_2}$ since $w_0s_2\in w_0\mathcal{C}_2$. If $w_1\in\Gamma(w_2)$, we will have $V(L_{w_1})=\mathcal{V}(w_1)\subseteq  V(L_{w_2})\subseteq  \fru_{\operatorname{J}_2}$. Then from Proposition \ref{Same}, $-w_0s_1\rho=\rho-\alpha_1$ will be  $\Phi_{{\rm J}_2}^+$-dominant, which is a contradiction since ${\langle\rho-\alpha_1,\alpha_1^\vee\rangle}=-1$. Thus $\Gamma(w_2)=\{w_2\}$. Now assume that for $t\leq i-1$, we always have $\Gamma(w_{t})=\{w_{t}\}$. When $t=i$, $-w_0s_i\rho=-w_0(\rho-\alpha_i)=\rho-\alpha_{i}$ is $\Phi_{{\rm J}_{i}}^+$-dominant. Then $V(L_{w_i})=V(L_{w_0s_i})\subseteq  \fru_{\operatorname{J}_i}$ since $w_0s_i\in w_0\mathcal{C}_i$. If $w_k=w_0s_ks_{k+1}\cdots s_{n-1}\in\Gamma(w_i)$ for some $1\leq k<i$, we will have $V(L_{w_k})=\mathcal{V}(w_k)\subseteq  V(L_{w_i})\subseteq \fru_{\operatorname{J}_i}$. Then from Proposition \ref{Same}, $-w_0s_k\rho=\rho-\alpha_k$ will be  $\Phi_{{\rm J}_i}^+$-dominant, which is a contradiction since ${\langle\rho-\alpha_{k},\alpha_{k}^\vee\rangle}=-1$. Thus $w_k\notin \Gamma(w_i)$. So we must have $\Gamma(w_i)=\{w_i\}$.

\item When $n$ is even  and $n-i< i$, we have $w_0(\alpha_i)=-\alpha_i$. In this case $w_i=w_0s_is_{i-1}\cdots s_{1}$ is a minimal-length element in the right cell $w_0\mathcal{C}_i$. By Proposition \ref{bruhat-order}, there are $n-i+1$ elements which may satisfy the condition (\ref{orbital-element}): $w_i,w_{i+1},\cdots, w_n=w_0s_ns_{n-2}\cdots s_1$, where $w_t=w_0s_ts_{t-1}\cdots s_1\in w_0\mathcal{C}_t$ for $i\leq t\leq n-1$. First we have $-w_0s_i\rho=-w_0(\rho-\alpha_i)=\rho-\alpha_{i}$, which is $\Phi_{{\rm J}_{i}}^+$-dominant. Then $V(L_{w_i})=V(L_{w_0s_i})\subseteq  \fru_{\operatorname{J}_i}$ since $w_0s_i\in w_0\mathcal{C}_i$. Similarly we have $ \Gamma(w_n)=\{w_n\}$ and $ \Gamma(w_{n-1})=\{w_{n-1}\}$ by Lemma \ref{min-length-w} since $w_n$ and $w_{n-1}$  are minimal-length elements in the two-sided cell $w_0\mathcal{C}$. If $t=n-2$, by Proposition \ref{bruhat-order}, there are only three elements which may satisfy the condition (\ref{orbital-element}): $w_{n}$, $w_{n-1}$ and $w_{n-2}$. If $w_{n-1}\in\Gamma(w_{n-2})$, we will have $V(L_{w_{n-1}})=\mathcal{V}(w_{n-1})\subseteq  V(L_{w_{n-2}})\subseteq  \fru_{\operatorname{J}_{n-2}}$. Then from Proposition \ref{Same}, $-w_0s_{n-1}\rho=\rho-\alpha_{n-1}$ will be  $\Phi_{{\rm J}_{n-2}}^+$-dominant, which is a contradiction since ${\langle\rho-\alpha_{n-1},\alpha_{n-1}^\vee\rangle}=-1$. Similarly we have $w_{n}\notin \Gamma(w_{n-2})$. Thus $\Gamma(w_{n-2})=\{w_{n-2}\}$. Then  we start the induction process from $t=n-3$ and assume that for $i+1\leq t\leq n$, we always have $\Gamma(w_{t})=\{w_{t}\}$.
When $t=i$, $-w_0s_i\rho=-w_0(\rho-\alpha_i)=\rho-\alpha_{i}$ is $\Phi_{{\rm J}_{i}}^+$-dominant. Then $V(L_{w_i})=V(L_{w_0s_i})\subseteq  \fru_{\operatorname{J}_i}$ since $w_0s_i\in w_0\mathcal{C}_i$. If $w_k\in\Gamma(w_i)$ for some $i+1\leq k\leq n$, we will have $V(L_{w_k})=\mathcal{V}(w_k)\subseteq  V(L_{w_i})\subseteq \fru_{\operatorname{J}_i}$. Then from Proposition \ref{Same}, $-w_0s_k\rho=\rho-\alpha_k$ will be  $\Phi_{{\rm J}_i}^+$-dominant, which is a contradiction since ${\langle\rho-\alpha_{k},\alpha_{k}^\vee\rangle}=-1$. Thus $w_k\notin \Gamma(w_i)$. So we must have $\Gamma(w_i)=\{w_i\}$.


\item When $n$ is odd, we will have
\[
w_0(\alpha_i)=-\alpha_{k_i}=\begin{cases}
    -\alpha_i, & \text{~if~} 1\leq i\leq n-2,\\
    -\alpha_n,& \text{~if~} i=n-1,\\
    -\alpha_{n-1}, &\text{~if~} i=n.
\end{cases}
\]

In this case, the argument is similar to the case when $n$ is even, so we omit the details here. Then we obtain $\Gamma(w_i) = \{w_i\}$.

 \end{enumerate}

 {On the other hand, we have
\[
\langle-w_0s_i\rho,\alpha_{k_i}^\vee\rangle
=\langle\rho-\alpha_{k_i},\alpha_{k_i}^\vee\rangle=-1,
\quad
\langle-w_0s_i\rho,\alpha_j^\vee\rangle>0
\quad(j\ne k_i).
\]
Then by Lemma \ref{minimal-integralequal} we have $V(L_w)=V(L_{w_0s_i})=\mathcal V(s_{k_i}w_0)=\mathcal V(w_0s_i)$ for $w\in w_0\mathcal{C}_i$.
}

\end{proof}

For the KL right cell $w_0\mathcal C_i$, we choose the following
minimal-length representative:
 $$w_{\rm min}=\begin{cases}
w_0s_1\cdots s_{n-2}s_{n-1} & \text{if~}i=1;\\
w_0s_is_{i+1}\cdots s_{n-1}  & \text{if~} 2\leq i\leq n-2\text{~and~}n-i\geq i;\\
w_0s_is_{i-1}\cdots s_{1}  & \text{if~} 2\leq i\leq n-2\text{~and~}n-i< i;\\
w_0s_{n-1}s_{n-2}\cdots s_1& \text{if~}i=n-1;\\
w_0s_{n}s_{n-2}\cdots s_1 & \text{if~}i=n.
\end{cases} $$

\begin{Rem}\label{low-rank-D}
Theorem \ref{AV-min}(2) also holds in types $D_3$ and $D_4$. The case $D_3$ is obvious. For type $D_4$, we have
\begin{align}\label{D-4}
  V(L_w)=\mathcal{V}(w_0s_i)=\mathcal{V}(w_{\min,i})\qquad\text{for all }w\in w_0\mathcal{C}_i,  
\end{align}
{
where we choose the minimal-length representatives
\[
\begin{aligned}
w_{\min,1}&=w_0s_1s_2s_3,&
w_{\min,2}&=w_0s_2s_3,\\
w_{\min,3}&=w_0s_3s_2s_1,&
w_{\min,4}&=w_0s_4s_2s_1.
\end{aligned}
\]
}
Then \eqref{D-4} holds for $i=1,3,4$ since 
$w_{\min,1}\in w_0\mathcal{C}_1$, $w_{\min,3}\in w_0\mathcal{C}_3$, $w_{\min,4}\in w_0\mathcal{C}_4$ are smooth. 

For $i=2$, we fix $0\ne e_{\alpha_2}\in\mathfrak g_{\alpha_2}$. Since $w_0=-1$ in type $D_4$, we have $\mathcal{V}(w_0s_2)=\mathcal{V}(s_2w_0)=\overline{B(\mathfrak{n}\cap{(w_0s_2)(\mathfrak n)})}=\overline{Be_{\alpha_2}}$ since $\mathfrak{n}\cap{(w_0s_2)(\mathfrak n)}
=\mathfrak{g}_{\alpha_2}$. Set $\lambda=-w_{\min,2}\rho=s_2s_3\rho
=\rho-2\alpha_2-\alpha_3$. Hence
\[
  \bigl(\langle\lambda,\alpha_j^\vee\rangle\bigr)_{j=1}^{4}
  =(3,-2,1,3),
\]
so $I_\lambda=\Pi\setminus\{\alpha_2\}$.
By Lemma \ref{minimal-integralequal} we have $V(L_{w_0s_2s_3})=V(L_{w_{\min,2}})=\overline{Be_{\alpha_2}}$. Also by Lemma \ref{gamma-w} we have $\mathcal{V}(w_{\min,2})\subseteq V(L_{w_{\min,2}})=\overline{Be_{\alpha_2}}$. On the other hand, $w_{\min,2}^{-1}(\alpha_2) = s_3 s_2 w_0(\alpha_2) = \alpha_2 + \alpha_3 \in \Phi^+$, it means that $e_{\alpha_2} \in \mathfrak{n} \cap {w_{\min,2}(\mathfrak n)}$ and hence $\overline{B e_{\alpha_2}} \subseteq \overline{B(\mathfrak{n} \cap {w_{\min,2}(\mathfrak n)})} = \mathcal{V}(w_{\min,2})$. Then we have 
$$V(L_w)=\mathcal{V}(s_2w_0)=V(L_{w_{\min,2}})=\mathcal{V}(w_{\min,2}),$$
\end{Rem}

\section{The case of $E$ and $F_4$}\label{typeef}

{For types $E$ and $F_4$, we have verified by Python that all the elements in $w_0\mathcal{C}$ are non-smooth, i.e.,  $S(w_0\mathcal{C}_i)=\varnothing$.}

The algorithm is based on the following steps.
\begin{enumerate}
    \item For $w\in W$, compute
    $\Phi_w=\Phi^+\cap w(\Phi^-)$
    by \cite[\S5.6, Exer.~1]{hum90}.

    \item Enumerate candidate positive generating sets $S$
    for subsystems of types $A_3,D_4$ in type $E$, and of types
    $B_2,B_3,C_3$ in type $F_4$.
     For each subsystem $\Delta$,
    determine the simple roots of $\Delta^+=\Delta\cap\Phi^+$
    by \cite[Lem. 3.1]{BGWX}.

    \item The set
    $\Delta_\sigma=\Phi_w\cap\Delta=\Phi_w\cap\Delta^+$
    is the inversion set of a unique
    $\sigma=f_\Delta(w)\in W(\Delta)$ by
    \cite[Lem.~2.1]{BP-05}.

    \item Recover $\sigma$ from this inversion set  by the standard procedure
  in  \cite[\S5.6, Exer.~1]{hum90}.
    Choose a simple root $\gamma_{i_1}$ of the induced positive
    system $\Delta^+$ which belongs to $\Delta_\sigma$, and set
    \[
       I_1=s_{\gamma_{i_1}}
       \bigl(\Delta_\sigma\setminus\{\gamma_{i_1}\}\bigr)
       =\Delta_{s_{\gamma_{i_1}}\sigma}.
    \]
    As long as $I_{r-1}\ne\varnothing$, choose a simple root
    $\gamma_{i_r}\in I_{r-1}$ and put
    \[
       I_r=s_{\gamma_{i_r}}
       \bigl(I_{r-1}\setminus\{\gamma_{i_r}\}\bigr)
       =\Delta_{s_{\gamma_{i_r}}\cdots s_{\gamma_{i_1}}\sigma}.
    \]
    Each step removes one root. Thus $I_k=\varnothing$ for some $k$,
    and $\sigma=s_{\gamma_{i_1}}\cdots s_{\gamma_{i_k}}$.

    \item Compare $\sigma$ with the forbidden elements in
    Theorem~\ref{BP-patterns}. A match in any retained flat subsystem proves
    that $X(w)$ is non-smooth. If all flat stellar subsystems have
    been tested and no match occurs, then $X(w)$ is smooth.
\end{enumerate}



\subsection{The case of $E$}
First we consider the case of type $E$. We have the following result.

\begin{prop}\label{e6-av}
    When $W=W(E_6)$ and $w\in w_0\mathcal{C}_i$ for $1\leq i\leq 6$, we have
    \[
V(L_w)=
\begin{cases}
\mathcal V(s_6w_0)=\mathcal V(w_0s_1)=\mathcal V(w_0s_1s_3s_4s_5s_6) & \text{if~}i=1;\\
\mathcal V(s_2w_0)=\mathcal V(w_0s_2)=\mathcal V(w_0s_2s_4s_5s_6) & \text{if~}i=2;\\
\mathcal V(s_5w_0)=\mathcal V(w_0s_3)=\mathcal V(w_0s_3s_4s_5s_6) & \text{if~}i=3;\\
\mathcal V(s_4w_0)=\mathcal V(w_0s_4)=\mathcal V(w_0s_4s_5s_6) & \text{if~}i=4;\\
\mathcal V(s_3w_0)=\mathcal V(w_0s_5)=\mathcal V(w_0s_5s_4s_3s_1) & \text{if~}i=5;\\
\mathcal V(s_1w_0)=\mathcal V(w_0s_6)=\mathcal V(w_0s_6s_5s_4s_3s_1) & \text{if~}i=6.
\end{cases}
\]
When $W = W(E_r)$ with $r=7,8$ and $w \in w_0 \mathcal{C}_i$ for $1 \le i \le r$, we have $V(L_w)=\mathcal V(s_iw_0)=\mathcal{V}(w_0s_i)=\mathcal{V}(w_{\rm min})$,
where $w_{\rm min}=w_i$ is the representative fixed in
Table~\ref{tab:E678-Bruhat-candidates}.

\end{prop}

\begin{proof}
    For $w_0x\in w_0\mathcal{C}$ we have $\ell(w_0x)=\ell(w_0)-\ell(x)$ (see \cite[Cor. 2.3.3]{BB05}). Thus  $\ell(w_0x)$ is minimal if and only if  $\ell(x)$ is maximal. Now we have two maximal length elements in the two-sided cell $\mathcal{C}$: $s_1s_3s_4s_5s_6\in \mathcal{C}_1$ and $s_6s_5s_4s_3s_1\in \mathcal{C}_6$. Thus from  Lemma \ref{min-length-w}, we have $V(L_w)=\mathcal V(w_0s_1s_3s_4s_5s_6)$ when $w\in w_0\mathcal{C}_1$ and $V(L_w)=\mathcal V(w_0s_6s_5s_4s_3s_1)$ when $w\in w_0\mathcal{C}_6$.


When $w \in w_0\mathcal{C}_3$, Lemma \ref{constant} implies $V(L_w) = V(L_{w_3})$, where  $w_3=w_0s_3s_4s_5s_6 \in w_0\mathcal{C}_3$. Note that $w_3$ is a minimal-length element in  the right cell $w_0\mathcal{C}_3$. From Lemma \ref{gamma-w} we have $V(L_{w_3})=\bigcup\limits_{y\in \Gamma(w_3)}\mathcal{V}(y)$ for some subset $\Gamma(w_3)\subseteq  W$. For any $y\in \Gamma(w_3)$, we  have \begin{equation}\label{orbital-ele}
    y\leq w_3 \text{~and~} y\in w_0\mathcal{C}
\end{equation} by Lemma \ref{gamma-w}. Proposition \ref{bruhat-order} implies that there are only two elements which may satisfy the condition (\ref{orbital-ele}): $w_3$, $w_1=w_0s_1s_3s_4s_5s_6\in w_0\mathcal{C}_1$. {Write $w_0(\alpha_i)=-\alpha_{k_i}$; with our numbering,
$(k_1,k_2,k_3,k_4,k_5,k_6)=(6,2,5,4,3,1)$. Put
${\rm J}_l=\Pi\setminus\{\alpha_l\}$. Then
$-w_0s_i\rho=\rho-\alpha_{k_i}$ is
$\Phi_{{\rm J}_{k_i}}^+$-dominant.}
Since $w_0s_i\in w_0\mathcal{C}_i$, Proposition \ref{Same} gives
$V(L_{w_0s_i})\subseteq\mathfrak{u}_{{\rm J}_{k_i}}$ for all $i$.
{The element $w_1$ is of minimal length in the two-sided cell
$w_0\mathcal{C}$, so Lemma \ref{min-length-w} gives
$V(L_{w_1})=\mathcal{V}(w_1)$.} If $w_1\in\Gamma(w_3)$, then
\[
V(L_{w_0s_1})=V(L_{w_1})
=\mathcal{V}(w_1)
\subseteq V(L_{w_3})
\subseteq\mathfrak{u}_{{\rm J}_5}.
\]
The first equality follows from Lemma \ref{constant}.
Proposition \ref{Same} would then imply that
$-w_0s_1\rho=\rho-\alpha_6$ is $\Phi_{{\rm J}_5}^+$-dominant,
contrary to
${\langle\rho-\alpha_6,\alpha_6^\vee\rangle}=-1$.
Thus $w_1\notin\Gamma(w_3)$ and
$\Gamma(w_3)=\{w_3\}$. 

The Bruhat candidate set for $w_2=w_0s_2s_4s_5s_6$ consists
only of $w_2$, so $\Gamma(w_2)=\{w_2\}$.

For $w_5=w_0s_5s_4s_3s_1$, the only other
candidate is $w_6=w_0s_6s_5s_4s_3s_1$. Since $w_6$ is a
minimal-length element of the two-sided cell, Lemma
\ref{min-length-w} gives $V(L_{w_6})=\mathcal{V}(w_6)$. If
$w_6\in\Gamma(w_5)$, then, by Lemma \ref{constant} we have
\[
V(L_{w_0s_6})=V(L_{w_6})
=\mathcal{V}(w_6)\subseteq V(L_{w_5})
=V(L_{w_0s_5})\subseteq\mathfrak{u}_{{\rm J}_3}.
\]
Proposition \ref{Same} {would imply} that
$-w_0s_6\rho=\rho-\alpha_1$
 is $\Phi_{{\rm J}_3}^+$-dominant, contrary to
${\langle\rho-\alpha_1,\alpha_1^\vee\rangle}=-1$. Thus $w_6\notin\Gamma(w_5)$ and
$\Gamma(w_5)=\{w_5\}$.

When $w \in w_0\mathcal{C}_4$, Lemma \ref{constant} implies $V(L_w) = V(L_{w_4})$, where  $w_4=w_0s_4s_5s_6 \in w_0\mathcal{C}_4$. Note that $w_4$ is a minimal-length element in  the right cell $w_0\mathcal{C}_4$. From Lemma \ref{gamma-w} we have $V(L_{w_4})=\bigcup\limits_{y\in \Gamma(w_4)}\mathcal{V}(y)$ for some subset $\Gamma(w_4)\subseteq  W$.
Similarly, the only possible elements of $\Gamma(w_4)$ are
$w_4,w_3,w_2,w_1$, as recorded in the table below. We have already
shown that $V(L_{w_i})=\mathcal{V}(w_i)$ for $i=1,2,3$.
If one of these $w_i$ belonged to $\Gamma(w_4)$, then by Lemma \ref{constant} we have
\[
V(L_{w_0s_i})=V(L_{w_i})=\mathcal{V}(w_i)
\subseteq V(L_{w_4})\subseteq\mathfrak{u}_{{\rm J}_4}.
\]
Proposition \ref{Same} would then imply that
$-w_0s_i\rho=\rho-\alpha_{k_i}$ is $\Phi_{{\rm J}_4}^+$-dominant.
This is impossible, since $k_i\in\{6,2,5\}$ and
${\langle\rho-\alpha_{k_i},\alpha_{k_i}^\vee\rangle}=-1$.
Thus $\Gamma(w_4)=\{w_4\}$.

Moreover, it is easy to verify that ${\langle-w_0s_i\rho,\alpha_{k_i}^\vee\rangle}={\langle\rho-\alpha_{k_i},\alpha_{k_i}^{\vee}\rangle}=-1$ and ${\langle-w_0s_i\rho,\alpha_j^\vee\rangle}>0$ for all $j\neq k_i$. Consequently, by Lemma \ref{minimal-integralequal}, we obtain $V(L_w)=V(L_{w_0s_i})=\mathcal V(s_{k_i}w_0)=\mathcal V(w_0s_i)$ for any $w\in w_0\mathcal{C}_i$.

For type $E_7$ and $E_8$, the arguments are similar and we omit the process here. Now the proof is finished.

\end{proof}

As the arguments in Proposition \ref{e6-av}  contain many repetitive steps, we omit some of them and list in  Table \ref{tab:E678-Bruhat-candidates} the set of elements satisfying Lemma \ref{gamma-w} in each case. In Table \ref{tab:E678-Bruhat-candidates}, $w_i$ denotes a minimal-length element in $w_0\mathcal C_i$.

\begingroup
\small
\setlength{\tabcolsep}{5pt}
\renewcommand{\arraystretch}{1.18}
\setlength{\LTleft}{\fill}
\setlength{\LTright}{\fill}

\begin{longtable}{@{}cclcc@{}}\caption{Bruhat candidates for $w_i$ in types $E_6$, $E_7$, and $E_8$.}
\label{tab:E678-Bruhat-candidates}\\
\toprule
Type & $i$ & $w_i$ & $\{z\in w_0\mathcal C\mid z\le w_i\}$
& \text{Candidates to exclude}  \\
\midrule
\endfirsthead

\toprule
Type & $i$ & $w_i$ & $\{z\in w_0\mathcal C\mid z\le w_i\}$
& \text{Candidates to exclude}  \\
\midrule
\endhead

\midrule
\endfoot

\endlastfoot

\multirow[c]{6}{*}{$E_6$}
& $1$ & $w_1={w_0s_1s_3s_4s_5s_6}$ & $\{w_1\}$ & $\varnothing$  \\*
& $2$ & $w_2={w_0s_2s_4s_5s_6}$ & $\{w_2\}$ & $\varnothing$   \\*
& $3$ & $w_3= {w_0s_3s_4s_5s_6}$ & $\{w_1,w_3\}$ & $w_1$  \\*
& $4$ & $w_4= {w_0s_4s_5s_6}$ & $\{w_1,w_2,w_3,w_4\}$ & $w_1,w_2,w_3$  \\*
& $5$ & $w_5= {w_0s_5s_4s_3s_1}$ & $\{w_5,w_6\}$ & $w_6$  \\*
& $6$ & $w_6= {w_0s_6s_5s_4s_3s_1}$ & $\{w_6\}$ & $\varnothing$  \\
\addlinespace[4pt]
\midrule

\multirow[c]{7}{*}{$E_7$}
& $1$ & $w_1= {w_0s_1s_3s_4s_5s_6s_7}$ & $\{w_1\}$ & $\varnothing$  \\*
& $2$ & $w_2= {w_0s_2s_4s_5s_6s_7}$ & $\{w_2\}$ & $\varnothing$  \\*
& $3$ & $w_3= {w_0s_3s_4s_5s_6s_7}$ & $\{w_1,w_3\}$ & $w_1$  \\*
& $4$ & $w_4= {w_0s_4s_5s_6s_7}$ & $\{w_1,w_2,w_3,w_4\}$ & $w_1,w_2,w_3$  \\*
& $5$ & $w_5= {w_0s_5s_4s_3s_1}$ & $\{w_5,w_6,w_7\}$ & $w_6,w_7$  \\*
& $6$ & $w_6= {w_0s_6s_5s_4s_3s_1}$ & $\{w_6,w_7\}$ & $w_7$  \\*
& $7$ & $w_7= {w_0s_7s_6s_5s_4s_3s_1}$ & $\{w_7\}$ & $\varnothing$ \\
\addlinespace[4pt]

\multirow[c]{8}{*}{$E_8$}
& $1$ & $w_1= {w_0s_1s_3s_4s_5s_6s_7s_8}$ & $\{w_1\}$ & $\varnothing$  \\*
& $2$ & $w_2= {w_0s_2s_4s_5s_6s_7s_8}$ & $\{w_2\}$ & $\varnothing$  \\*
& $3$ & $w_3= {w_0s_3s_4s_5s_6s_7s_8}$ & $\{w_1,w_3\}$ & $w_1$  \\*
& $4$ & $w_4= {w_0s_4s_5s_6s_7s_8}$ & $\{w_1,w_2,w_3,w_4\}$ & $w_1,w_2,w_3$  \\*
& $5$ & $w_5= {w_0s_5s_4s_3s_1}$ & $\{w_5,w_6,w_7,w_8\}$ & $w_6,w_7,w_8$  \\*
& $6$ & $w_6= {w_0s_6s_5s_4s_3s_1}$ & $\{w_6,w_7,w_8\}$ & $w_7,w_8$  \\*
& $7$ & $w_7= {w_0s_7s_6s_5s_4s_3s_1}$ & $\{w_7,w_8\}$ & $w_8$  \\*
& $8$ & $w_8= {w_0s_8s_7s_6s_5s_4s_3s_1}$ & $\{w_8\}$ & $\varnothing$ \\

\bottomrule

\end{longtable}
\endgroup

\subsection{The case of $F_4$}\label{typef4}
Now we consider the case of type $F_4$. We have the following result.

\begin{prop}\label{f4av}
    When $W=W(F_4)$ and $w\in w_0\mathcal{C}_i$, we have
    \[
V(L_w)=
\begin{cases}
\mathcal V(w_0s_1s_2s_3s_2s_1) & \text{if~}i=1;\\
\mathcal V(w_0s_2s_3s_2s_1) & \text{if~}i=2;\\
\mathcal V(w_0s_3s_2s_3s_4) & \text{if~}i=3;\\
\mathcal V(w_0s_4s_3s_2s_3s_4) & \text{if~}i=4.
\end{cases}
\]
\end{prop}
\begin{proof}
    For $w_0w\in w_0\mathcal{C}$ we have $\ell(w_0w)=\ell(w_0)-\ell(w)$ (see \cite[Cor. 2.3.3]{BB05}). Thus  $\ell(w_0w)$ is minimal if and only if  $\ell(w)$ is maximal. The desired conclusion for $i=1$ and $4$ now follows from Lemma \ref{min-length-w} since we have two maximal length elements in the two-sided cell $\mathcal{C}$: $s_1s_2s_3s_2s_1\in \mathcal{C}_1$ and $s_4s_3s_2s_3s_4\in \mathcal{C}_4$.


When $w \in w_0\mathcal{C}_2$, Lemma \ref{constant} implies $V(L_w) = V(L_{w_2})$, where  $w_2=w_0s_2s_3s_2s_1 \in w_0\mathcal{C}_2$. Note that $w_2$ is a minimal-length element in  the right cell $w_0\mathcal{C}_2$. From Lemma \ref{gamma-w} we have $V(L_{w_2})=\bigcup\limits_{y\in \Gamma(w_2)}\mathcal{V}(y)$ for some subset $\Gamma(w_2)\subseteq  W$. For any $y\in \Gamma(w_2)$, we  have $y\leq w_2$ and $y\in w_0\mathcal{C}$ by Lemma \ref{gamma-w}. Proposition \ref{bruhat-order} implies that there are only two elements which may be contained in $\Gamma(w_2)$: $w_2$, $w_1=w_0s_1s_2s_3s_2s_1\in w_0\mathcal{C}_1$. Denote ${\rm J}_i=\Pi\setminus \{\alpha_i\}$.
Since $w_0=-1$, we have
$-w_0s_i\rho=\rho-\alpha_i$, which is
$\Phi_{{\rm J}_i}^+$-dominant. Thus Proposition \ref{Same} gives
$V(L_{w_0s_i})\subseteq\mathfrak{u}_{{\rm J}_i}$.
{The element $w_1$ is of minimal length in the two-sided cell
$w_0\mathcal{C}$, so Lemma \ref{min-length-w} gives
$V(L_{w_1})=\mathcal{V}(w_1)$.} If $w_1\in\Gamma(w_2)$, then
{
\[
V(L_{w_0s_1})=V(L_{w_1})
=\mathcal{V}(w_1)
\subseteq V(L_{w_2})
\subseteq\mathfrak{u}_{{\rm J}_2}.
\]
}
{The first equality follows from Lemma \ref{constant}.}
Proposition \ref{Same} would then imply that $\rho-\alpha_1$ is
$\Phi_{{\rm J}_2}^+$-dominant, contrary to
${\langle\rho-\alpha_1,\alpha_1^\vee\rangle}=-1$.
Thus $w_1\notin\Gamma(w_2)$ and
$\Gamma(w_2)=\{w_2\}$. 

When $w \in w_0\mathcal{C}_3$, Lemma \ref{constant} implies $V(L_w) = V(L_{w_3})$, where  $w_3=w_0s_3s_2s_3s_4 \in w_0\mathcal{C}_3$. Note that $w_3$ is a minimal-length element in  the right cell $w_0\mathcal{C}_3$. From Lemma \ref{gamma-w} we have $V(L_{w_3})=\bigcup\limits_{y\in \Gamma(w_3)}\mathcal{V}(y)$ for some subset $\Gamma(w_3)\subseteq  W$. For any $y\in \Gamma(w_3)$, we  have $y\leq w_3$ and $y\in w_0\mathcal{C}$ by Lemma \ref{gamma-w}. Proposition \ref{bruhat-order} implies that the only possible elements
of $\Gamma(w_3)$ are $w_3$ and
$w_4=w_0s_4s_3s_2s_3s_4\in w_0\mathcal{C}_4$.
{The element $w_4$ has also  minimal length in the two-sided
cell $w_0\mathcal{C}$, so Lemma \ref{min-length-w} gives
$V(L_{w_4})=\mathcal{V}(w_4)$.} If $w_4\in\Gamma(w_3)$, then
Lemma \ref{constant} gives
{
\[
V(L_{w_0s_4})=V(L_{w_4})
=\mathcal{V}(w_4)
\subseteq V(L_{w_3})
\subseteq\mathfrak{u}_{{\rm J}_3}.
\]
}
Proposition \ref{Same} would imply that
$-w_0s_4\rho=\rho-\alpha_4$ is
$\Phi_{{\rm J}_3}^+$-dominant, contrary to
${\langle\rho-\alpha_4,\alpha_4^\vee\rangle}=-1$.
Thus $w_4\notin\Gamma(w_3)$ and
$\Gamma(w_3)=\{w_3\}$. 
Now the proof is finished.
\end{proof}

The Bruhat candidates arising in the proof of Proposition \ref{f4av} are summarized in Table \ref{tab:F4-wi-candidates}. 

\begingroup
\small
\setlength{\tabcolsep}{5pt}
\renewcommand{\arraystretch}{1.18}
\setlength{\LTleft}{\fill}
\setlength{\LTright}{\fill}

\begin{longtable}{@{}cclcc@{}}
\caption{ {Bruhat candidates for $w_i$ in type $F_4$.}}
\label{tab:F4-wi-candidates}\\
\toprule
{Type} & {$i$} & {$w_i$} &
{$\{z\in w_0\mathcal C\mid z\le w_i\}$}
& {Candidates to exclude} \\
\midrule
\endfirsthead

\toprule
{Type} & {$i$} &  {$w_i$} &
 {$\{z\in w_0\mathcal C\mid z\le w_i\}$}
&  {Candidates to exclude} \\
\midrule
\endhead

\midrule
\endfoot

\endlastfoot

\multirow[c]{4}{*}{ {$F_4$}}
&  {$1$} &
 {$w_1=w_0s_1s_2s_3s_2s_1$} &
 {$\{w_1\}$} &
 {$\varnothing$} \\*
&  {$2$} &
 {$w_2=w_0s_2s_3s_2s_1$} &
 {$\{w_1,w_2\}$} &
 {$w_1$} \\*
&  {$3$} &
 {$w_3=w_0s_3s_2s_3s_4$} &
 {$\{w_3,w_4\}$} &
 {$w_4$} \\*
&  {$4$} &
 {$w_4=w_0s_4s_3s_2s_3s_4$} &
 {$\{w_4\}$} &
 {$\varnothing$} \\

\bottomrule

\end{longtable}
\endgroup

\section{The case of $G_2$}\label{typeg2}

Let $\Phi$ be the root system of type $G_2$, with simple reflections $s_1=s_{\alpha_1}$ ($ \alpha_1$ is the short root) and $s_2=s_{\alpha_2}$ ($\alpha_2$ is the long root). The Weyl group $W(G_2)$ has $12$ elements. Let $w_0$ be the longest element of $W(G_2)$, with length $\ell(w_0)=6$. It is easy to see that $w_0\mathcal{C}_1=\mathcal{C}_2$ and $w_0\mathcal{C}_2=\mathcal{C}_1$.


For $G_2$, from Theorem \ref{BP-patterns}, the only possible stellar subsystems are those of rank $\le 2$:\[\Delta = B_2\text{ and } \Delta = G_2.\]
Thus we only need to check $\Delta = G_2$ (the whole system) and $\Delta = B_2$ (if it exists inside $G_2$).

When $\Delta = G_2$, $f_\Delta(w) = w = w_0 s$. From Theorem \ref{BP-patterns}, the forbidden patterns for $G_2$ are:\[[s_2]s_1 s_2 s_1 [s_2],\quad s_1 s_2 s_1 s_2 s_1.\]


 When $\Delta = B_2$. The length ratio in $G_2$ is $1:\sqrt{3}$, while in $B_2$ it is $1:\sqrt{2}$. Thus no $B_2$ subsystem exists in $G_2$ since root lengths are not compatible. Hence no check needed for $B_2$.

Therefore,  for $s \in W(G_2)$, the Schubert variety $X(w_0 s)$ is smooth if and only if \[w_0 s \notin \{ s_1 s_2 s_1,\ s_2 s_1 s_2 s_1,\ s_1 s_2 s_1 s_2,\ s_2 s_1 s_2 s_1 s_2,\ s_1 s_2 s_1 s_2 s_1 \}.\]

\begin{cor}
For type $G_2$,     we have $S(w_0\mathcal{C}_{1})=S(\mathcal{C}_{2})=\{s_2,s_2s_1,s_2s_1s_2\}$ and $S(w_0\mathcal{C}_{2})=S(\mathcal{C}_{1})=\{s_1,s_1s_2\}$.
\end{cor}

\begin{cor}
    When $W=W(G_2)$ and $w\in w_0\mathcal{C}=\mathcal{C}$, the associated variety of $L_w$ will be irreducible. We have $V(L_w)=\mathcal{V}(s_i)$ if  $w\in \mathcal{C}_{i}$.
\end{cor}

\section{The minimal-length representatives $w_{\rm min}$ of right cells $w_0\mathcal{C}_i$ }\label{minimal-rep}

Let
\[
  \mathcal{M}_i
  :=\bigl\{w\in w_0\mathcal C_i\mid
       \ell(w)=\min_{y\in w_0\mathcal C_i}\ell(y)\bigr\}.
\]
The following table lists every element of $\mathcal{M}_i$.
 For $D_3$, we retain
the $D$-type labels, so that vertex $1$ is adjacent to vertices $2$ and $3$.

\begingroup
\small
\renewcommand{\arraystretch}{1.24}
\setlength{\tabcolsep}{5pt}
\begin{longtable}{@{}p{31mm}p{28mm}p{62mm}c@{}}
\caption{All minimal-length elements in the right cells
  $w_0\mathcal C_i$.}
\label{tab:minimal-length-right-cells}\\
\toprule
Type & Cell index $i$ & $\mathcal{M}_i$ & $|\mathcal{M}_i|$ \\
\midrule
\endfirsthead
\multicolumn{4}{@{}l}{\tablename~\thetable\ (continued)}\\
\toprule
Type & Cell index $i$ & $\mathcal{M}_i$ & $|\mathcal{M}_i|$ \\
\midrule
\endhead
\midrule
\multicolumn{4}{r@{}}{Continued on the next page}\\
\endfoot
\bottomrule
\endlastfoot

$A_1$ & $1$ & $\{e\}$ & $1$ \\
\addlinespace[0.5ex] 
\hdashline
\addlinespace[0.5ex]
\multirow[c]{3}{*}{$A_n$ $(n\geq2)$}   
 & $2i<n+1$
 & $\{w_0s_is_{i+1}\cdots s_n\}$ & $1$ \\
 & $2i>n+1$
 & $\{w_0s_is_{i-1}\cdots s_1\}$ & $1$ \\
 & $2i=n+1$
 & $\{w_0s_is_{i-1}\cdots s_1,\;w_0s_is_{i+1}\cdots s_n\}$
 & $2$ \\
\midrule
$B_n,C_n$ $(n\geq2)$
 & $1\leq i<n$
 & $\{w_0s_is_{i+1}\cdots s_ns_{n-1}\cdots s_1\}$ & $1$ \\
$B_2,C_2$ & $2$
 & $\{w_0s_2s_1s_2\}$ & $1$ \\
$B_3,C_3$ & $3$
 & $\{w_0s_3s_2s_1,\;w_0s_3s_2s_3\}$ & $2$ \\
$B_n,C_n$ $(n\geq4)$ & $n$
 & $\{w_0s_ns_{n-1}\cdots s_1\}$ & $1$ \\
\midrule
\multirow[c]{6}{*}{$D_n$ $(n\geq5)$}
 & $\substack{1\leq i\leq n-2\\2i<n}$
 & $\left\{\begin{gathered}
      w_0s_is_{i+1}\cdots s_{n-2}s_{n-1},\\
      w_0s_is_{i+1}\cdots s_{n-2}s_n
    \end{gathered}\right\}$ & $2$ \\
 & $\substack{1\leq i\leq n-2\\2i=n}$
 & $\left\{\begin{gathered}
      w_0s_is_{i-1}\cdots s_1,\\
      w_0s_is_{i+1}\cdots s_{n-2}s_{n-1},\\
      w_0s_is_{i+1}\cdots s_{n-2}s_n
    \end{gathered}\right\}$ & $3$ \\
 & $\substack{1\leq i\leq n-2\\2i>n}$
 & $\{w_0s_is_{i-1}\cdots s_1\}$ & $1$ \\
 & $n-1$
 & $\{w_0s_{n-1}s_{n-2}\cdots s_1\}$ & $1$ \\
 & $n$
 & $\{w_0s_ns_{n-2}s_{n-3}\cdots s_1\}$ & $1$ \\
\addlinespace[0.5ex] 
\hdashline
\addlinespace[0.5ex]
\multirow[c]{4}{*}{$D_4$} & $1$
 & $\{w_0s_1s_2s_3,\;w_0s_1s_2s_4\}$ & $2$ \\
 & $2$
 & $\{w_0s_2s_1,\;w_0s_2s_3,\;w_0s_2s_4\}$ & $3$ \\
 & $3$
 & $\{w_0s_3s_2s_1,\;w_0s_3s_2s_4\}$ & $2$ \\
 & $4$
 & $\{w_0s_4s_2s_1,\;w_0s_4s_2s_3\}$ & $2$ \\
\addlinespace[0.5ex] 
\hdashline
\addlinespace[0.5ex]
\multirow[c]{3}{*}{$D_3$} & $1$
 & $\{w_0s_1s_2,\;w_0s_1s_3\}$ & $2$ \\
 & $2$ & $\{w_0s_2s_1s_3\}$ & $1$ \\
 & $3$ & $\{w_0s_3s_1s_2\}$ & $1$ \\
 \midrule
\multirow[c]{6}{*}{$E_6$} & $1$ & $\{w_0s_1s_3s_4s_5s_6\}$ & $1$ \\
 & $2$ & $\{w_0s_2s_4s_3s_1,\;w_0s_2s_4s_5s_6\}$ & $2$ \\
 & $3$ & $\{w_0s_3s_4s_5s_6\}$ & $1$ \\
 & $4$ & $\{w_0s_4s_3s_1,\;w_0s_4s_5s_6\}$ & $2$ \\
 & $5$ & $\{w_0s_5s_4s_3s_1\}$ & $1$ \\
 & $6$ & $\{w_0s_6s_5s_4s_3s_1\}$ & $1$ \\
\midrule
\multirow[c]{7}{*}{$E_7$} & $1$ & $\{w_0s_1s_3s_4s_5s_6s_7\}$ & $1$ \\
 & $2$ & $\{w_0s_2s_4s_5s_6s_7\}$ & $1$ \\
 & $3$ & $\{w_0s_3s_4s_5s_6s_7\}$ & $1$ \\
 & $4$ & $\{w_0s_4s_5s_6s_7\}$ & $1$ \\
 & $5$ & $\{w_0s_5s_4s_3s_1\}$ & $1$ \\
 & $6$ & $\{w_0s_6s_5s_4s_3s_1\}$ & $1$ \\
 & $7$ & $\{w_0s_7s_6s_5s_4s_3s_1\}$ & $1$ \\
\midrule
\multirow[c]{8}{*}{$E_8$} & $1$ & $\{w_0s_1s_3s_4s_5s_6s_7s_8\}$ & $1$ \\
 & $2$ & $\{w_0s_2s_4s_5s_6s_7s_8\}$ & $1$ \\
 & $3$ & $\{w_0s_3s_4s_5s_6s_7s_8\}$ & $1$ \\
 & $4$ & $\{w_0s_4s_5s_6s_7s_8\}$ & $1$ \\
 & $5$ & $\{w_0s_5s_4s_3s_1,\;w_0s_5s_6s_7s_8\}$ & $2$ \\
 & $6$ & $\{w_0s_6s_5s_4s_3s_1\}$ & $1$ \\
 & $7$ & $\{w_0s_7s_6s_5s_4s_3s_1\}$ & $1$ \\
 & $8$ & $\{w_0s_8s_7s_6s_5s_4s_3s_1\}$ & $1$ \\
\multirow[c]{4}{*}{$F_4$} & $1$ & $\{w_0s_1s_2s_3s_2s_1\}$ & $1$ \\
 & $2$ & $\{w_0s_2s_3s_2s_1\}$ & $1$ \\
 & $3$ & $\{w_0s_3s_2s_3s_4\}$ & $1$ \\
 & $4$ & $\{w_0s_4s_3s_2s_3s_4\}$ & $1$ \\
\midrule
\multirow[c]{2}{*}{$G_2$} & $1$ & $\{w_0s_1s_2s_1s_2s_1\}=\{s_2\}$ & $1$ \\
 & $2$ & $\{w_0s_2s_1s_2s_1s_2\}=\{s_1\}$ & $1$ \\
\end{longtable}
\endgroup

Indeed, since $\ell(w_0x)=\ell(w_0)-\ell(x)$ for every $x\in W$,
\[
  \mathcal{M}_i =\bigl\{w_0x\mid x\in\mathcal C_i,\;\ell(x)=\max_{y\in\mathcal C_i}\ell(y)\bigr\}.
\]
Thus the table follows by selecting all longest words in the descriptions
of the right cells $\mathcal C_i$ in Lemma~\ref{rightcell-ci}. In type $A_n$, the two candidate
lengths are $i$ and $n-i+1$. In types $B_n$ and $C_n$, the candidate lengths are $i$ and $2n-i$
for $i<n$; for $i=n$, they are $n$ and $3$. In type $D_n$, they are
$i,n-i,n-i$ for $i\leq n-2$, and $n-1,3$ for each of the two end vertices.
Whenever the maximal length is attained by several distinct words, all of
them are included.






\end{document}